\pdfoutput=1
\documentclass{article}
\usepackage{excludeonly}

\usepackage[draft]{todonotes}

\usepackage{savesym}
\usepackage{amsthm,amsmath}
\usepackage{amssymb,latexsym,graphicx}
\usepackage{amscd}
 \usepackage[all,cmtip]{xy}
\usepackage{tikz-cd}
\usepackage{accents}
\usepackage{cite}
\usepackage{mathtools}
\usepackage{stmaryrd} % provides \llbracket (formal power series)

\usepackage{calligra}
\DeclareMathAlphabet{\mathcalligra}{T1}{calligra}{m}{n}
\DeclareMathAlphabet{\mathpzc}{OT1}{pzc}{m}{it}
\usepackage{hycolor}
\usepackage{xcolor}
\usepackage[
                       colorlinks=true,
                       linkcolor=black, %red,
                       citecolor=black, %green,
                       urlcolor=blue,
                     ]{hyperref}
\usepackage{soul} %provides strikethrough-command st{.. text ..}

\usepackage{makeidx}
\usepackage[intoc]{nomentbl}
\makenomenclature

\usepackage{stackengine} %%% provides Nequation - named and numbered equation
\usepackage{booktabs} % tables

\newtheorem{theorem}{Theorem}[section]
\newtheorem{corollary}[theorem]{Corollary}

\newtheorem{lemma}[theorem]{Lemma}
\newtheorem{proposition}[theorem]{Proposition}

\theoremstyle{definition}
\newtheorem{definition}[theorem]{Definition}

\newtheorem{remark}[theorem]{Remark}

\newtheorem{example}[theorem]{Example}

\theoremstyle{remark}

\newcommand{\C}{{\mathbb{C}}}

\newcommand{\N}{{\mathbb{N}}}

\newcommand{\R}{{\mathbb{R}}}
\renewcommand{\SS}{{\mathbb{S}}}

\newcommand{\Z}{{\mathbb{Z}}}
\newcommand{\Aa}{{\mathcal{A}}}   % connections
\newcommand{\Bb}{{\mathcal{B}}}
\newcommand{\Ee}{{\mathcal{E}}}

\newcommand{\Ll}{{\mathcal{L}}}   % Lagrangian planes
\newcommand{\im}{{\rm im\, }}             % image
\newcommand{\SPAN}{{\rm span\, }}         % span
\newcommand{\id}{{\rm id}}                % identity
\newcommand{\Id}{{\rm Id}}
\newcommand{\Map}{{\rm Map}}          % Maps

\newcommand{\norm}{{\rm norm}}

\newcommand{\eps}{{\varepsilon}}

\newcommand{\inner}[2]{\langle #1, #2\rangle}   
\newcommand{\INNER}[2]{\left\langle #1, #2\right\rangle}

\newcommand{\SC}{{\mathrm{sc}}}
\newcommand{\SCz}{{\mathrm{sc}^0}}
\newcommand{\SCo}{{\mathrm{sc}^1}}
\newcommand{\EM}{{\mathrm{M}}}

\newcommand{\mbf}[1]{\text{\boldmath $#1$}}  % mbf=mathboldface

\def\NABLA#1{{\mathop{\nabla\kern-.5ex\lower1ex\hbox{$#1$}}}}
\def\Nabla#1{\nabla\kern-.5ex{}_{#1}}
\def\Tabla#1{\Tilde\nabla\kern-.5ex{}_{#1}}
\def\abs#1{\mathopen|#1\mathclose|}   
\def\Abs#1{\left|#1\right|}            
\def\norm#1{\mathopen\|#1\mathclose\|}
\def\Norm#1{\left\|#1\right\|}

\renewcommand{\Tilde}{\widetilde}

\newcommand{\INTO}{\hookrightarrow}              % embedding
\newlength\eqshift
\renewcommand\theequation{\thesection.\arabic{equation}}
\let\savetheequation\theequation

\makeatletter
\renewcommand*\env@matrix[1][\arraystretch]{%
  \edef\arraystretch{#1}%
  \hskip -\arraycolsep
  \let\@ifnextchar\new@ifnextchar
  \array{*\c@MaxMatrixCols c}}
\makeatother

\makeatletter
\let\save@mathaccent\mathaccent
\newcommand*\if@single[3]{%
  \setbox0\hbox{${\mathaccent"0362{#1}}^H$}%
  \setbox2\hbox{${\mathaccent"0362{\kern0pt#1}}^H$}%
  \ifdim\ht0=\ht2 #3\else #2\fi
  }
\newcommand*\rel@kern[1]{\kern#1\dimexpr\macc@kerna}
\newcommand*\widebar[1]{\@ifnextchar^{{\wide@bar{#1}{0}}}{\wide@bar{#1}{1}}}
\newcommand*\wide@bar[2]{\if@single{#1}{\wide@bar@{#1}{#2}{1}}{\wide@bar@{#1}{#2}{2}}}
\newcommand*\wide@bar@[3]{%
  \begingroup
  \def\mathaccent##1##2{%
    \let\mathaccent\save@mathaccent
    \if#32 \let\macc@nucleus\first@char \fi
    \setbox\z@\hbox{$\macc@style{\macc@nucleus}_{}$}%
    \setbox\tw@\hbox{$\macc@style{\macc@nucleus}{}_{}$}%
    \dimen@\wd\tw@
    \advance\dimen@-\wd\z@
    \divide\dimen@ 3
    \@tempdima\wd\tw@
    \advance\@tempdima-\scriptspace
    \divide\@tempdima 10
    \advance\dimen@-\@tempdima
    \ifdim\dimen@>\z@ \dimen@0pt\fi
    \rel@kern{0.6}\kern-\dimen@
    \if#31
      \overline{\rel@kern{-0.6}\kern\dimen@\macc@nucleus\rel@kern{0.4}\kern\dimen@}%
      \advance\dimen@0.4\dimexpr\macc@kerna
      \let\final@kern#2%
      \ifdim\dimen@<\z@ \let\final@kern1\fi
      \if\final@kern1 \kern-\dimen@\fi
    \else
      \overline{\rel@kern{-0.6}\kern\dimen@#1}%
    \fi
  }%
  \macc@depth\@ne
  \let\math@bgroup\@empty \let\math@egroup\macc@set@skewchar
  \mathsurround\z@ \frozen@everymath{\mathgroup\macc@group\relax}%
  \macc@set@skewchar\relax
  \let\mathaccentV\macc@nested@a
  \if#31
    \macc@nested@a\relax111{#1}%
  \else
    \def\gobble@till@marker##1\endmarker{}%
    \futurelet\first@char\gobble@till@marker#1\endmarker
    \ifcat\noexpand\first@char A\else
      \def\first@char{}%
    \fi
    \macc@nested@a\relax111{\first@char}%
  \fi
  \endgroup
}
\makeatother

\long\def\symbolfootnote[#1]#2{\begingroup%
\def\thefootnote{\fnsymbol{footnote}}\footnote[#1]{#2}\endgroup}
\begin{document}
\sloppy
%\author{Urs Frauenfelder \qquad Joa Weber
%                 Instituto de Matem\'{a}tica, Estat\'{\i}stica
%                 e Computa\c{c}\~{a}o Scient\'{\i}fica \\
%                 Universidade Estadual de Campinas \\
%                 Rua S\'{e}rgio Buarque de Holanda~651,
%                 Cidade Universit\'{a}ria "Zeferino Vaz",
%                CEP~13083-859, Campinas-SP, Brasil.
           %     \\
%                 joa@math.sunysb.edu.
%              \\
%              Tel.: +55-19-3521xxxx\\
%              Fax: +55-19-3521xxxx\\
%            }

\author{\quad Urs Frauenfelder \quad \qquad\qquad
             Joa Weber%\footnote{
%        {\bf Financial support:}
%        Funda\c{c}\~{a}o de Amparo
%        \`{a} Pesquisa do Estado de S\~{a}o Paulo
%        (FAPESP), processo $\mathrm{n}^{\rm o}$ 2017/19725-6,
        %and CNPq, Conselho Nacional de Desenvolvimento Cient\'{\i}fico
        %e Tecnol\'ogico - Brasil.
        %Bolsista do CNPq - Brasil.
        %%%\hfill
        %  se publicado individualmente:
        % "O presente trabalho foi realizado com apoio do CNPq, Conselho Nacional de Desenvol          vimento Científico  e Tecnológico - Brasil".
        %  se publicado em co-autoria:
        % "Bolsista do CNPq - Brasil".
        %
        %%%\newline
%        {\bf Address:}
%        Instituto de Matem\'{a}tica, Estat\'{\i}stica
%        e Computa\c{c}\~{a}o Scient\'{\i}fica,
%        Universidade Estadual de Campinas,
%        Rua S\'{e}rgio Buarque de Holanda~651,
        %SP~13083-859 ,
%        Campinas, SP, Brasil.
        % MSC 37Dxx 58E05
        %
%        \hfill
%        joa@ime.unicamp.br
%  }
    \\
    Universit\"at Augsburg \qquad\qquad%\quad
    UNICAMP
}

\title{The shift map on Floer trajectory spaces}

%\subtitle{-- Monograph --}  %%% Springer book style only
\date{\today}

%\begin{titlepage}
\maketitle %(to set the title page and copyright page; see note)
                 %\include files (e.g., preface, introduction)
%\thispagestyle{empty}
%\newpage
%
%{\color{red}
%  \subsection*{To do}
%  \begin{itemize}
%  \item
%    ...
%  \end{itemize}
%}

%\end{titlepage}

%%%%%%%%%%%%%%%%%%%%%%%%%%%%%%%%%%
%%%%%%%%%% FRONTMATTER %%%%%%%%%%%%%
%%%%%%%%%%%%%%%%%%%%%%%%%%%%%%%%%%
%\frontmatter
%\include{dedic}
%\include{foreword}
%\include{preface}
%\include{acknow}
%
%\tableofcontents
%
%\include{acronym}

%\frontmatter %• title page and copyright page information
%\include{0-dedic}
% dedication by hand:
%     \clearpage\thispagestyle{empty}
%      \par\vspace*{.35\textheight}{\centering Dedicated to ...\par}\clearpage
%
%\include{0-preface}

% ACKNOWLEDGEMENTS %
% * various anonymous referees?
% * DISCLAIMER: - references based on authors knowledge
%                        - more people contributed etc etc

%%%%%%%%%%%%%%%%%%%%%%%%%%%%%%%%%%
%%%%%%% main matter %%%%%%%%%%%%%%%%%%
%%%%%%%%%%%%%%%%%%%%%%%%%%%%%%%%%%
%\mainmatter
%\include{part}
%\include{chapter}
%\include{appendix}

%\mainmatter %\include files (e.g., main chapters, appendices)

% introduction
%\cleardoublepage
%\phantomsection
%\include{1_sc-smoothness}      % INTRODUCTION

%%%%%%%%%%%%%%%%%%%%%%%%%%%%%%%%%%%
%%%%%%% Abstract %%%%%%%%%%%%%%%%%%%%%
%%%%%%%%%%%%%%%%%%%%%%%%%%%%%%%%%%%
\begin{abstract}
In this article we give a uniform proof
why the shift map on Floer homology trajectory spaces
is scale smooth.
This proof works for various Floer homologies, periodic,
Lagrangian, Hyperk\"ahler, elliptic or parabolic, and uses
Hilbert space valued Sobolev theory.
\end{abstract}

\tableofcontents

%.\newpage
%%%%%%%%%%%%%%%%%%%%%%%%%%%%%%%%%%%
%%%%%%% Chapter:  %%%%%%%%%%%%%%%
%%%%%%%%%%%%%%%%%%%%%%%%%%%%%%%%%%%
\section{Introduction}

In Morse and Floer homology gradient flow lines
play a crucial role.
These locally lie in spaces of maps
$\R\to S$ of the real line into a vector space.
The shift map $\Psi:\R\times \Map(\R,S)\to \Map(\R,S)$ is
defined by $(\tau,v)\mapsto \tau_* v$, where $(\tau_*v)(t)=v(\tau+t)$
for $t\in\R$. After endowing the mapping space with some topology
the shift map has terrible properties.
\newline
By differentiating the shift map with respect to the first variable
one looses a derivative. Moreover, the shift map is merely continuous
in the compact open topology, but not in the norm topology.
In the last two decades this led Hofer, Wysocki, and
Zehnder~\cite{Hofer:2017a} to the discovery and exploration of a new
notion of smoothness in infinite dimensions called
\emph{scale smoothness} or \emph{$\SC$-smoothness}.
See also the article by Fabert et al~\cite{Fabert:2016a}
for the crucial importance of the new notion in various
fields of symplectic topology or the article by 
Hofer~\cite{Hofer:2017c} for a survey.

For Morse homology the vector space $S$ is finite
dimensional. However, for Floer homology the vector space $S$
will be infinite dimensional. 

Floer theory arises in many features.
In the study of periodic orbits of Hamiltonian systems
one looks at a closed string version of Floer homology,
namely, periodic Floer homology~\cite{floer:1988c,floer:1989a}.
In this case the vector space consists of loops
in a finite dimensional symplectic vector space $S$.
The study of gradient flow lines is based on the study of an elliptic
pde on the cylinder.

In the study of Lagrangian intersection points
one looks at an open string analogon, namely
Floer homology with Lagrangian boundary conditions~\cite{floer:1988a}.
In this case the vector space $S$ consists of paths
in a symplectic vector space which start and end
in a Lagrangian subspace.
The study of gradient flow lines is based on the study of an elliptic
pde on the strip.

The second author established Morse homology for the
heat flow~\cite{weber:2013a,weber:2013b}
which led to the study of a parabolic pde on the cylinder.
This was an essential ingredient in the joint proof
with Dietmar Salamon~\cite{salamon:2006a}
of the famous Viterbo isomorphism~\cite{Viterbo:1998a}.

Motivated by Donaldson-Thomas gauge theory
in higher dimensions~\cite{donaldson:1998a}
Hohloch, Noetzel, and Salamon~\cite{hohloch:2009a}
discovered a hyperk\"ahler version of Floer homology
which leads to dynamics in higher dimensional time,
see as well Ginzburg and Hein~\cite{Ginzburg:2012a}.
In this setup the vector space $S$ consists of maps
from a three-dimensional closed manifold into
hyperk\"ahler space.

Although the shift map has terrible properties in
the first variable, it is quite innocent in the second variable,
namely, it is \emph{linear}. People familiar with Sobolev
theory~\cite{Adams:2003a} know what a pain products are.
Thanks to linearity in the second variable
this difficulty is absent.

It is generally believed that the moduli space of (unparametrized)
gradient flow lines, namely gradient trajectories modulo shift,
can be interpreted as the zero set of a section from an $\SC$-manifold
into an $\SC$-bundle over it.
If the gradient flow lines are allowed to be broken, then the
$\SC$-manifold has to be replaced by an $\EM$-polyfold.
In the Morse case this is explained by Albers and
Wysocki~\cite{Albers:2013b}.
See as well Wehrheim~\cite{Wehrheim:2012b}
for the case of periodic Floer homology.

A crucial ingredient to construct this
$\SC$-manifold is the scale smoothness of the shift map.
In view of the various Floer homologies defined on different spaces
displaying elliptic and parabolic features one might worry
that this property has to be proved for each Floer theory
individually.
The purpose of this paper is to give a uniform
treatment of the scale smoothness of the shift map
which is applicable for all the above mentioned Floer homologies.
This crucially uses the linearity of the shift map
in the second variable.

In order to formulate this uniform treatment we use
Hilbert space valued Sobolev spaces.
This idea is not new in Floer homology, but has
sucessfully been used by Robbin and Salamon\cite{robbin:1995a}
in the treatment of the spectral flow.
Moreover, this gives interesting connections between
Floer homology and interpolation theory.
For a detailled treatment of interpolation theory
see e.g.~\cite{Triebel:1978a}.

The philosophy behind this comes from scale Morse homology,
namely, a still to be developped Morse homology on scale manifolds.
Scale Morse homology not only gives a unified treatment
of various topics in Floer homology like gluing,
spectral flow, and exponential decay~\cite{albers:2013a},
but due to its non-local character scale Morse homology
leads to so far unexplored applications to delay equations,
as discussed by the first author jointly with
Peter Albers and Felix Schlenk~\cite{Albers:2018b,Albers:2018c}.

\vspace{.2cm}
\textit{This paper is organized as follows.}
In Section~\ref{sec:shift-loop} we explain that the shift map
is continuous in the compact open topology, but fails to be continuous
in the operator topology.
The strange behavior of the shift map was one of the main
reasons for Hofer, Wysocki, and Zehnder to introduce
scale smoothness whose definition we recall in
Section~\ref{sec:scale-smoothness}.
In order to explain scale smooth one needs the notion
of a scale structure which we recall in Section~\ref{sec:scale-structures}.
In Section~\ref{sec:scale-structures} we also introduce the examples
of scale structures which are relevant in Floer homology.
The examples use Hilbert valued Sobolev theory.
The proof that these examples actually satisfy the axioms
of a scale structure is carried out in Section~\ref{sec:Bsc-str-ex}.
The importance of having a scale structure is that this
guarantees that the chain rule holds.
We recall the chain rule in Section~\ref{sec:chain-rule}.
In Section~\ref{sec:scale-actions}
we give a uniform proof that the shift map is $\SC$-smooth
for the trajectory spaces relevant in the various types of Floer
homologies.
The reason why the trajectory spaces
which we introduce in Section~\ref{sec:scale-structures}
is explained in Section~\ref{sec:mapping-spaces}.

\vspace{.2cm}
\textit{Acknowledgements.}
This article was written during the stay of the first author
at the Universidade Estadual de Campinas (UNICAMP) whom the first
author would like to thank for hospitality.
His visit was supported by UNICAMP and Funda\c{c}\~{a}o de Amparo
\`{a} Pesquisa do Estado de S\~{a}o Paulo (FAPESP),
processo $\mathrm{n}^{\rm o}$ 2017/19725-6.

%.\newpage
%%%%%%%%%%%%%%%%%%%%%%%%%%%%%%%%%%%
%%%%%%% Chapter:  %%%%%%%%%%%%%%%
%%%%%%%%%%%%%%%%%%%%%%%%%%%%%%%%%%%
%\chapter{Shift map on Floer trajectory spaces}

%.\newpage
%%%%%%%%%%%%%%%%%%%%%%%%%%%%%%%%%%%
%%%%%%% Section: Shift Map %%%%%%%%%%%%%%%
%%%%%%%%%%%%%%%%%%%%%%%%%%%%%%%%%%%
\section{Shift map on loop spaces}\label{sec:shift-loop}

Throughout we identify the unit circle $\SS^1$ with the quotient space
$\R/\Z$. To indicate that a function $v:\R\to\R$ has the property of being
$1$-periodic, that is $v(t+1)=v(t)$ for every $t\in\R$, we use the
notation $v:\SS^1\to\R$.

As a warmup we discuss in this section 
the shift map on the loop space $H=L^2(\SS^1,\R)$.
For $\tau\in\SS^1$ define the \textbf{shift map}
\[
     \Psi_\tau:H\to H,\quad v\mapsto \tau_* v
\]
where $(\tau_* v)(t):=v(t+\tau)$. Observe that $\Psi_\tau$ is linear
and an isometry
\[
     \Norm{\Psi_\tau v}_{H}=\norm{v}_H,\quad
     \tau\in\SS^1,\, v\in H.
\]

\begin{lemma}[Continuity in compact open topology]
\label{le:shift-C0-compopen}
As $\tau$ goes to zero, $\Psi_\tau$ converges to
$\Psi_0=\id$ in the \textbf{compact open topology}, i.e.
for each $v\in H$ it holds that
\[
     \lim_{\tau\to 0}\Norm{\Psi_\tau(v)-v}_H=0.
\]
\end{lemma}

\begin{proof}
Since $C^\infty(\SS^1,\R)$ is dense in $L^2(\SS^1,\R)$ there is a
sequence $v_\nu\in C^\infty(\SS^1,\R)$ such that $v_\nu\to v$ in $L^2$.
Choose $\eps>0$.
Because $v_\nu$ converges to $v$ in $L^2$,
there is $\nu_0\in \N$ such that
\begin{equation}\label{eq:1}
     \norm{v_\nu-v}_{H}\le\eps/3
\end{equation}
whenever $\nu\ge\nu_0$
Since $v_{\nu_0}$ is smooth and uniformly continuous
there is a time $\tau_0>0$ such that for all $\tau\in[0,\tau_0]$
it holds that
\[
     \Abs{v_{\nu_0}(t+\tau)-v_{\nu_0}(t)}\le\eps/3,\quad t\in\SS^1.
\]
We estimate for $0\le\tau\le\tau_0$
\begin{equation}\label{eq:2}
\begin{split}
     \Norm{\Psi_\tau(v_{\nu_0})-v_{\nu_0}}_{H}
   &=\sqrt{\int_0^1\Abs{v_{\nu_0}(t+\tau)-v_{\nu_0}(t)}^2\, dt}\\
   &\le \sqrt{\int_0^1\left(\eps/3\right)^2\, dt}\\
   &=\eps/3.
\end{split}
\end{equation}
Combining~(\ref{eq:1}) and~(\ref{eq:2}) we estimate
\begin{equation*}%\label{eq:}
\begin{split}
     \Norm{\Psi_\tau(v)-v}_{H}
   &\le \Norm{\Psi_\tau(v)-\Psi_\tau(v_{\nu_0}) }_{H}
     +\Norm{\Psi_\tau(v_{\nu_0})-v_{\nu_0}}_{H}
     +\Norm{v_{\nu_0}-v}_{H}\\
   &=\Norm{v-v_{\nu_0}}_{H}
     +\Norm{\Psi_\tau(v_{\nu_0})-v_{\nu_0}}_{H}
     +\Norm{v-v_{\nu_0}}_{H}\\
   &=\eps
\end{split}
\end{equation*}
where the second step uses that $\Psi_\tau$ is linear and an isometry.
\end{proof}

\begin{lemma}[Discontinuity in norm topology]
As $\tau$ goes to zero, $\Psi_\tau$ does \underline{not} converge
to $\Psi_0=\id$ in the norm topology. More precisely,
for each $0<\tau\le 1/2$ there is an element $v_\tau\in H$ of norm $1$
and with the property that $\norm{\Psi_\tau(v_\tau)-v_\tau}_H=\sqrt{2}>0$.
\end{lemma}

\begin{proof}
As illustrated by Figure~\ref{fig:fig-Psi_v-v} we define
\begin{equation*}
   v_\tau(t):=
   \begin{cases}
     0&\text{, $t\in(0,1-\frac{1}{2}\tau)$.}\\
     \sqrt{2/\tau}&\text{, $t\in[1-\frac{\tau}{2},1]$,}
   \end{cases}
\end{equation*}
and
\begin{figure}%[h]
  \centering
  \includegraphics%[width=0.9\textwidth]
                             [height=4cm]
                             {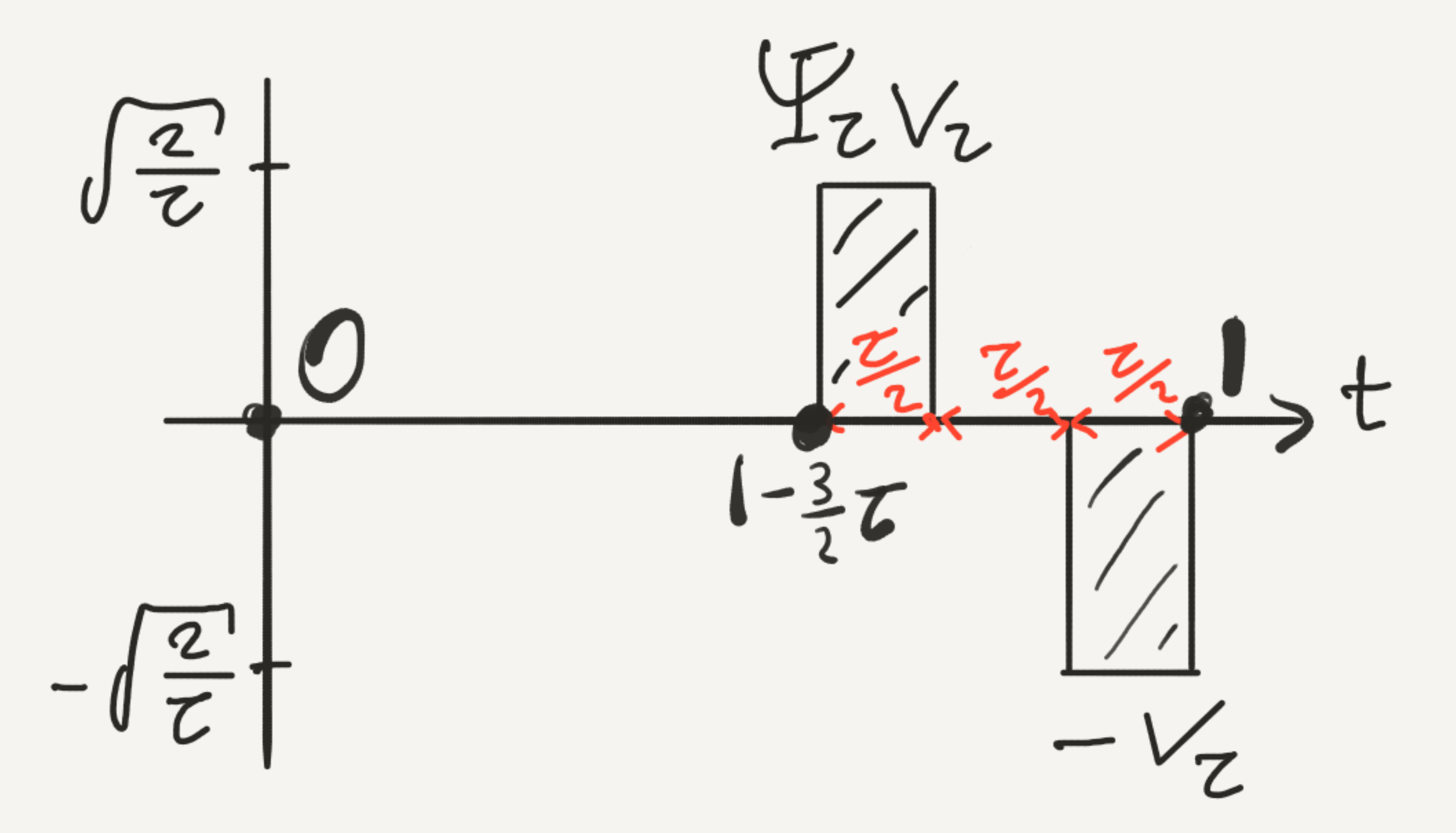}
  \caption{The function $\Psi_\tau(v_\tau)-v_\tau$}
  \label{fig:fig-Psi_v-v}
\end{figure}
compute
\begin{equation*}%\label{eq:}
\begin{split}
     \Norm{v_\tau}_H
     =\sqrt{\int_0^1\Abs{v_\tau(t)}^2\, dt}
     =\sqrt{\frac{2}{\tau}\,\frac{\tau}{2}}=1.
\end{split}
\end{equation*}
Note that
\begin{equation*}
   \left(\Psi_\tau(v_\tau)-v_\tau\right)(t):=
   \begin{cases}
     \sqrt{2/\tau}&\text{, $t\in[1-\frac{3}{2}\tau,1-\tau]$,}\\
     -\sqrt{2/\tau}&\text{, $t\in[1-\frac{\tau}{2},1]$,}\\
     0&\text{, else.}
   \end{cases}
\end{equation*}
Hence
\begin{equation*}
   \Abs{\left(\Psi_\tau(v_\tau)-v_\tau\right)(t)}^2:=
   \begin{cases}
     2/\tau&\text{, $t\in[1-\frac{3}{2}\tau,1-\tau]\cup[1-\frac{\tau}{2},1]$,}\\
     0&\text{, else.}\\
   \end{cases}
\end{equation*}
We calculate
\begin{equation*}%\label{eq:}
\begin{split}
     \Norm{\Psi_\tau(v_\tau)-v_\tau}_H
     =\sqrt{\int_0^1
     \Abs{\left(\Psi_\tau(v_\tau)-v_\tau\right)(t)}^2\, dt
     }
     =\sqrt{\frac{2}{\tau}\,(\tau/2+\tau/2)}=\sqrt{2}.
\end{split}
\end{equation*}
\end{proof}

%%%%%%%%%%%%%%%%%%%%%%%%%%%%%%%%%%%
%%%%%%% Section: Scale structures %%%%%%%%%%
%%%%%%%%%%%%%%%%%%%%%%%%%%%%%%%%%%%
\section{Scale structures}\label{sec:scale-structures}

Scale structures were introduced by Hofer, Wysocki, and
Zehnder~\cite{Hofer:2007a,Hofer:2017c,Hofer:2017a}.
We first recall its definition and then we discuss the
main examples relevant in Morse and Floer homology.
That these examples satisfy the conditions of a scale structure
is proved in Section~\ref{sec:Bsc-str-ex}.
The examples for Floer homology use Hilbert space
valued Sobolev theory motivated by the paper
of Robbin and Salamon on the spectral flow~\cite{robbin:1995a}.

Let $(E,\norm{\cdot}_E)$ be a Banach space.

\begin{definition}\label{def:Banach-scale}
A \textbf{scale-structure} on $E$, also called an
\textbf{sc-structure} or a \textbf{Banach scale}, is a nested sequence
$E=:E_0\supset E_1\supset E_2\supset\dots$ of Banach spaces
meeting the following requirements:
\begin{itemize}
\item[(i)]
   Each level includes compactly into the previous one, i.e. the
   linear operator given by inclusion
  $E_{i+1}\INTO E_i$ is compact for each $i\in\N_0$.
\item[(ii)]
  The intersection $E_\infty:=\cap_{i\ge 0} E_i$ is dense in each level $E_i$.
\end{itemize}
In this case one calls $E$ a scale Banach space
and also writes $E=(E_i)_{i\in{\N_0}}$.
\end{definition}

\begin{remark}\label{rem:dense}
a) It follows from~(ii) that the inclusions
$E_{i+1}\INTO E_i$ are dense for all $i\in\N_0$.
b) The intersection $E_\infty$ carries the structure of a Fr\'{e}chet space.
\end{remark}

\begin{definition}[Shifted scale Banach space]\label{def:shift-scBsp}
Given a scale Banach space $E$ and $m\in\N_0$, then one defines
the scale Banach space $E^m$ by
\[
     (E^m)_k:=E_{k+m}.
\]
\end{definition}

\begin{definition}[Scale direct sum]
If $E$ and $F$ are scale Banach spaces one defines their
direct sum as the scale Banach space $E\oplus F$
whose levels are given by
\[
     (E\oplus F)_k:=E_k\oplus F_k.
\]
\end{definition}

\begin{definition}[Scale isomorphism]
A map $I:E\to F$ between scale Banach spaces
is called a \textbf{scale morphism}, or an
\textbf{\boldmath$\SC$-morphism},
if the restriction to each level $E_k$
takes values in $F_k$ and
\[
     I_k:=I|_{E_k}:E_k\to F_k
\]
is linear and continuous.
A scale morphism
is called a \textbf{scale isomorphism}, or an
\textbf{\boldmath$\SC$-isomorphism},
if its restriction $I_k$ to each level $E_k$ is bijective.
Note that by the open mapping theorem
if $I$ is a scale isomorphism its inverse is a scale
isomorphism as well.
Two scale Banach spaces are called \textbf{scale isomorphic}
if there exists a scale isomorphism between them.
\end{definition}

%%%%%%% Subsection: Examples %%%%%%%%%%%%
%%%%%%%%%%%%%%%%%%%%%%%%%%%%%%%%%%%
\subsection*{Examples}

\begin{example}[Finite dimension]
If the Banach space $E$ is of finite dimension, then
property (ii) implies that the scale-structure is constant
$E=:E_0=E_1=E_2=\dots$.
\end{example}

\begin{remark}[Infinite dimension]
In contrast, if $E$ is infinite dimensional, then the compactness
requirement in property (i) enforces strict inclusions
$E_{i+1}\subsetneq E_i$. Indeed the identity map
on an infinite dimensional Banach space is never compact,
because the unit ball of a Banach space is compact
if and only if the Banach space is finite dimensional.
\end{remark}

Now we introduce the main examples of scale structures
relevant in this text. The proofs that these examples satisfy the
requirements of a scale structure are
given in Section~\ref{sec:Bsc-str-ex}.

\begin{example}[The fractal Hilbert scales $\ell^{2,f}$]
\label{eq:sc-ell2}
Given a monotone unbounded function $f:\N\to(0,\infty)$,
define the Hilbert space of weighted $\ell^2$-sequences by
\[
     \ell^2_f
     :=\left\{ x=(x_\nu)_{\nu\in\N}\in\ell^2\,\bigg|\,
     \sum_{\nu=1}^\infty f(\nu) x_\nu^2<\infty\right\}.
\]
The inner product on $\ell^2_f$ is given by
\[
     \INNER{x}{y}_f:=\sum_{\nu=1}^\infty f(\nu) x_\nu y_\nu.
\]
We obtain an sc-structure on $H=\ell^2$
by using the Hilbert spaces
\begin{equation}\label{eq:ell-2-f-k}
     H_k:=(\ell^{2,f})_k:=\ell^2_{f^k},\quad k\in\N_0.
\end{equation}

Denote by $e_i=(0,\dots,0,1,0,\dots)$ the sequence whose members are
all $0$ except for member $i$ which is $1$.
The set $\Ee$ of all $e_i$ not only forms an orthonormal basis of the Hilbert
space $H_0=\ell^2$, but simultaneously an orthogonal basis of all
spaces $H_k=\ell^2_{f^k}$. Rescaling provides an orthonormal basis
\[
     \Ee_{f^k}:=\{e_{i, f^k}\mid i\in \N\},\quad
     e_{i, f^k}:=\frac{1}{f(i)^{k/2}}\, e_i ,
\]
of $H_k$. The isometric Hilbert space isomorphism obtained by
identifying the canonical orthonormal bases, namely
\[
     \phi_k:H_0\to H_k,\quad
     e_i\mapsto e_{i, f^k},
\]
induces a levelwise-isometric $\SC$-isomorphism
\[
     \phi_k:H^0\to H^k.
\]
This means that the restriction to the $i^{\rm th}$ level
\[
     \phi_k|_{(H^0)_i}: (H^0)_i=H_i\to
     (H^k)_i=H_{k+i}
\]
is an isometric Hilbert space isomorphism,
and this is true for every level $i\in \N_0$.
This explains the term \textbf{fractal}
since as a consequence each of the Banach scales $H^j$
is self-similar to any $H^k$.
The fractal scales $\ell^{2,f}$ are intensively studied in
interpolation theory; see e.g.~\cite{Triebel:1978a}.

Two monotone unbounded functions $f,g:\N\to(0,\infty)$
are called \textbf{equivalent}
if there is a constant $c>0$ such that $\frac{1}{c} f\le g\le c f$.
Two equivalent weight functions provide equivalent inner products
on the vector space $\ell^2_f=\ell^2_g$.

The \textbf{\boldmath product $f*g$} of two monotone unbounded
functions $f,g:\N\to(0,\infty)$
is the monotone unbounded function $f*g:\N\to(0,\infty)$
whose value at $\nu$ is the $\nu^{\rm th}$ smallest number among the
members of the two lists $(f(1),f(2),\dots)$ and $(g(1),g(2),\dots)$.
The Banach scale associated to $f*g$ is the scale direct sum
\[
     \ell^{2,f*g}=\ell^{2,f}\oplus\ell^{2,g}.
\]
\end{example}

\begin{example}[Path spaces for Morse homology]
\label{eq:path-Morse}
Fix a monotone cutoff function $\beta\in C^\infty(\R,[-1,1])$
with $\beta(s)-1$ for $s\le-1$ and $\beta(s)=1$ for $s\ge 1$.
Fix a constant $\delta>0$ and, see Figure~\ref{fig:fig-Morse-exp-weight},
define a function $\gamma_\delta:\R\to\R$ by
\[
     \gamma_\delta(s):=e^{\delta\beta(s) s}.
\]
\begin{figure}%[h]
  \centering
  \includegraphics%[width=0.9\textwidth]
                             [height=4cm]
                             {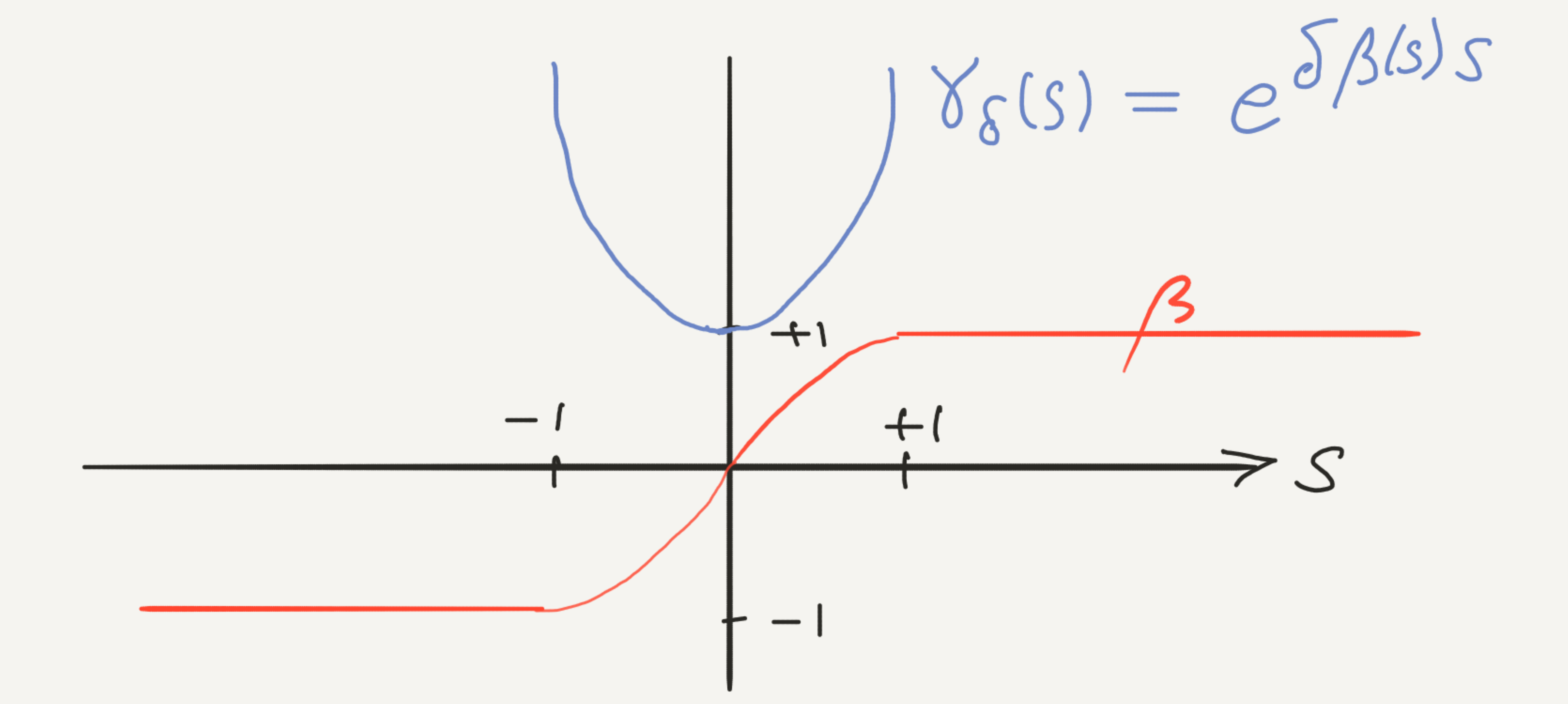}
  \caption{Monotone cutoff function $\beta$ and exponential weight $\gamma_\delta$}
  \label{fig:fig-Morse-exp-weight}
\end{figure}
Pick a constant $p\in(1,\infty)$.
Consider the Banach spaces defined for $k\in\N_0$ by
\begin{equation}\label{eq:W-kp-delta}
     W^{k,p}_\delta(\R,\R^n)
     :=\{v\in W^{k,p}(\R,\R^n)\mid
     \gamma_\delta v\in W^{k,p}(\R,\R^n)\}
\end{equation}
with norm
\[
     \Norm{v}_{W^{k,p}_\delta}
     :=\Norm{\gamma_\delta v}_{W^{k,p}}.
\]
Choose a sequence $0=\delta_0<\delta_1<\delta_2\dots$ and
define
\begin{equation}\label{eq:Morse-path-Ek}
     E_k:=W^{k,p}_{\delta_k}(\R,\R^n) ,\quad k\in\N_0.
\end{equation}
\end{example}

\begin{definition}[Weighted Hilbert space valued Sobolev spaces]
\label{def:weighted-spaces}
Let $k\in\N_0$, $p\in(1,\infty)$, and $\delta\ge 0$.
Suppose $H$ is a separable Hilbert space
and define the space $W^{k,p}_\delta(\R,H)$
by~(\ref{eq:W-kp-delta}) with $\R^n$ replaced by $H$.
This is again a Banach space, see Appendix~\ref{sec:Hilb-Sob}.
\end{definition}

\begin{example}[Path spaces for Floer homology]
\label{eq:path-Floer}
Pick a constant $p\in(1,\infty)$. For $k\in\N_0$
let $H_k$ be as in Example~\ref{eq:sc-ell2} and let $\delta_k$ be a
sequence as in Example~\ref{eq:path-Morse}.
The Banach space $E_k$ is defined as the intersection
of $k+1$ Banach spaces
\begin{equation}\label{eq:E_k-Floer}
     E_k:=\bigcap_{i=0}^k W^{i,p}_{\delta_k}(\R,H_{k-i}) ,\quad k\in\N_0.
\end{equation}
The norm on $E_k$ is defined by taking the maximum of the $k+1$
individual norms. This not only defines a norm, but even a complete one.
\end{example}

%\newpage.\newpage
%%%%%%%%%%%%%%%%%%%%%%%%%%%%%%%%%%%
%%%%%%% Section: Scale smoothness %%%%%%%%
%%%%%%%%%%%%%%%%%%%%%%%%%%%%%%%%%%%
\section{Scale smoothness}\label{sec:scale-smoothness}
The notion of scale smoothness is due to Hofer, Wysocki, and
Zehnder~\cite{Hofer:2007a,Hofer:2017c,Hofer:2017a}.
An elegant way to introduce scale smoothness is via
the tangent map.
In particular, the chain rule, see Section~\ref{sec:chain-rule},
is nicely explained using the tangent map.
We also give equivalent descriptions of scale smoothness
in terms of $\SC$-differentials. This equivalent description
is useful to check scale smoothness explicitly in examples.

Let $E$ be a scale Banach space.
Given an open subset $U\subset E$, then the part of $U$ in $E_k$
is denoted by $U_k:=U \cap E_k$. Note that $U_k$ is open in $E_k$.
In particular, one obtains a nested sequence
$U=U_0\supset U_1\supset U_2\dots$.

\begin{definition}[Scale continuity]
Suppose that $E$ and $F$ are sc-Banach spaces
and $U\subset E$ is an open subset.
A map $f:U\to F$ is called \textbf{scale continuous}
or of class $\mbf{\SCz}$ if
\begin{itemize}
\item[(i)]
  $f$ is level preserving, i.e. the restriction of $f$ to each level
  $U_k$ takes values in the corresponding level $F_k$, and
\item[(ii)]
  the map $f|_{U_k}:U_k\to F_k$ is continuous.
\end{itemize}
\end{definition}

In order to introduce the notion of scale differentiable or $\SCo$
we first need to introduce the notion of tangent bundle.
The \textbf{tangent bundle} of a scale Banach space $E$
is defined as the scale Banach space
\[
     TE:=E^1\oplus E^0.
\]
If $U\subset E$ is an open subset of the scale Banach space $E$,
as in Definition~\ref{def:shift-scBsp}
one denotes by $U^m\subset E^m$ the scale of open subsets
whose levels are given by $(U^m)_k:=U_{m+k}$ where $k\in\N_0$.
The tangent bundle of $U$ is the open subset
of $TE$ defined by
\[
     TU:=U^1\oplus E^0\subset TE.
\]
Note that the filtration of $TU$ is given by
\[
     (TU)_k=U_{k+1}\oplus E_k,\quad k\in\N_0.
\]

\begin{definition}[Scale differentiability]
Suppose $f:U\to F$ is $\SCz$, then $f$ is
called \textbf{continuously scale differentiable}
or of class $\mbf{\SCo}$ if for every $x\in U_1$
there is a bounded linear map
\[
     Df(x):E_0\to F_0,
\]
called \textbf{\boldmath$\SC$-differential},
such that the following two conditions hold:
\begin{itemize}
\item[(i)]
  The restriction of $f$ to $U_1$ interpreted as a map
  $f:U_1\to F_0$ is required to be \emph{pointwise} differentiable in
  the usual sense.
  The restriction of $Df(x)$ to $E_1$ is required to be
  the differential of $f:U_1\to F_0$
  in the usual sense, notation $df(x):E_1\to F_0$, i.e.
  \begin{equation}\label{eq:scale-differential}
     Df(x)|_{E_1}=df(x).
  \end{equation}
\item[(ii)]
  The \textbf{tangent map} $Tf:TU\to TF$
  defined for $(x,h)\in U^1\oplus E^0=TU$
  by
  \[
     Tf(x,h):=\left( f(x),Df(x)h\right)
  \]
  is of class $\SCz$.
\end{itemize}
\end{definition}

\begin{remark}[Unique extension and continuous differentiability]
Suppose $f:U\to F$ is of class $\SCo$.

a) Because $E_1$ is dense in $E_0$, see Remark~\ref{rem:dense},
the map $Df(x)$ is uniquely determined by~(\ref{eq:scale-differential}).
However, note that the mere requirement that $f:U_1\to F_0$ is
pointwise differentiable does not guarantee that a bounded extension
of $df(x):E_1\to F_0$ to $E_0$ exists. Existence of such an extension
is part of the definition of $\SCo$.

b) Because $E_1$ includes compactly in $E_0$, the usual differential
$df(x)\in\Ll(E_1,F_0)$ depends continuously on $x\in U_1$.
In other words, the restriction
\[
     f\in C^1(U_1,F_0)
\]
is not only pointwise, but even continuously, differentiable in the
usual sense.

To see this suppose $x_\nu\in U_1\subset E_1$ is a sequence of points
converging to a point $x\in U_1$.
Now assume by contradiction that there is a constant $\eps>0$
and a sequence in $E_1$ of unit norm $\norm{h_\nu}_{E_1}=1$
such that $\norm{df(x_\nu)h_\nu-df(x)h_\nu}_{F_0}\ge \eps$.
Now since the inclusion $E_1\INTO E_1$ is compact
there are subsequences, still denoted by $x_\nu$ and $h_\nu$,
such that
\begin{equation*}
\begin{split}
     \lim_{\nu\to\infty}\Norm{df(x_\nu)h_\nu-df(x)h_\nu}_{F_0}
   &=\lim_{\nu\to\infty}\Norm{Df(x_\nu)h_\nu-Df(x)h_\nu}_{F_0}\\
   &=\Norm{Df(x)h-Df(x)h_\nu}_{F_0}\\
   &=0.
\end{split}
\end{equation*}
Here step one uses that $Df(p)|_{E_1}=df(p)\in\Ll(E_1,F_0)$
for $p\in U_1$ and step two holds by continuity of $Df$
with respect to the compact open topology.
Hence $\eps=0$. Contradiction.
\end{remark}

\begin{remark}[Level preservation]\label{rem:level-preservation}
Suppose that $x$ lies in $U_m$ and $k\in\{0,\dots,m-1\}$.
Then, firstly, the restriction $Df(x)|_{E_k}:E_k\to F_0$ to $E_k$
automatically takes values in the Banach space $F_k$ and, secondly, the linear
operator
$$
     Df(x)|_{E_k}:E_k\to F_k
$$
is bounded. This follows from condition~(ii) of scale differentiability:
Because $x$ lies in $U_m$ and $k<m$, one has $x\in U_{k+1}$.
Since the tangent map is $\SCz$ it maps
\[
     (TU)_k=(U^1\oplus E^0)_k=U_{k+1}\oplus E_k
\]
to
\[
     (TF)_k=(F^1\oplus F^0)_k=F_{k+1}\oplus F_k
\]
continuously.
\end{remark}

\begin{remark}[Continuity in compact-open topology]
\label{rem:co-top}
A further consequence of condition (ii) in scale differentiability
is that the scale differential viewed as a map
$$Df|_{U_{k+1}\oplus E_k}: U_{k+1}\oplus E_k\to F_k$$
is continuous in the compact-open topology:
Namely, if $h\in E_k$ and $x_\nu$ is a sequence in $U_{k+1}$, then
\[
     \lim_{\nu\to 0}
     \Norm{Df(x_\nu)h-Df(x)h}_{F_k}=0.
\]
\end{remark}

\begin{lemma}[Characterization of $\SCo$ in terms of the scale-differential $Df$]
Assume that $f:U\to F$ is $\SCz$. Then
$f$ is $\SCo$ if and only if the following conditions hold:
\begin{itemize}
\item[(i)]
  $f:U_1\to F_0$ is pointwise differentiable in the usual sense;
\item[(ii)]
  For every $x\in U_1$ the differential $df(x):E_1\to F_0$
  has a continuous extension
  $Df(x):E_0\to F_0$;
\item[(iii)]
  For all $k\in\N_0$ and $x\in U_{k+1}$, the continuous extension
  $Df(x):E_0\to F_0$ restricts to a continuous map
  \[
     Df(x)|_{E_k}:E_k\to F_k
  \] 
  such that
  \[
     Df|_{U_{k+1}\oplus E_k}: U_{k+1}\oplus E_k\to F_k
  \]
  is continuous in the compact-open topology.
\end{itemize}
\end{lemma}

\begin{proof}
'$\Rightarrow$'
Suppose $f$ is $\SCo$. Then statements (i) and (ii) are obvious
and (iii) follows from Remarks~\ref{rem:level-preservation} and~\ref{rem:co-top}.

'$\Leftarrow$'
Suppose that $f$ is $\SCz$ and satisfies (i--iii). We have to show
that the tangent map is $\SCz$. We first discuss why $Tf$
maps $(TU)_k$ to $(TF)_k$ for every $k\in\N_0$.
Pick $(x,h)\in(TU)_k=U_{k+1}\oplus E_k$.
Since $f$ is $\SCz$ we have that $f(x)\in F_{k+1}$.
By (iii) we have that $Df(x)h\in F_k$.
Hence
\[
     Tf(x,h)=(f(x), Df(x)h)\in F_{k+1}\oplus F_k=(TF)_k.
\]
This shows that $Tf$ maps $(TU)_k$ to $(TF)_k$.

We next explain why $Tf$ as a map $Tf|_{(TU)_k}:(TU)_k\to(TF)_k$ is
continuous. Assume $(x_\nu,h_\nu)\in(TU)_k=U_{k+1}\oplus E_k$
is a sequence which converges to $(x,h)\in(TU)_k$.
Because $f$ is $\SCz$, it follows that
\[
     \lim_{\nu\to\infty} f(x_\nu)=f(x).
\]
Again by (iii) we have that 
\[
     \lim_{\nu\to\infty} Df(x_\nu)h_\nu=Df(x)h.
\]
Therefore
\[
     \lim_{\nu\to\infty} Tf(x_\nu,h_\nu)
     =\lim_{\nu\to\infty} \left(f(x_\nu),Df(x_\nu)h_\nu\right)
     =\left(f(x),Df(x)h\right)
     =Tf(x,h).
\]
This proves continuity and hence the lemma holds.
\end{proof}

For $m\ge 2$ one defines higher continuous scale differentiability
$\SC^m$ recursively as follows.

\begin{definition}[Higher scale differentiability]
An $\SCo$-map $f:U\to F$ is of class $\mbf{\SC^m}$ if and only if
its tangent map $Tf:TU\to TF$ is $\SC^{m-1}$.
It is called \textbf{sc-smooth}, or of class $\mbf{\SC^\infty}$,
iff it is of class $\SC^m$ for every $m\in\N$.
\end{definition}

An $\SC^m$-map has iterated tangent maps as follows.
Recursively one defines the iterated tangent bundle as
$T^{m+1}U:=T(T^m U)$.
For example
\begin{equation*}
\begin{split}
     T^2U=T(TU)&=T(U^1\oplus E^0)\\
   &=(U^1\oplus E^0)^1\oplus(E^1\oplus E^0)\\
   &=U^2\oplus E^1\oplus E^1\oplus E^0.
\end{split}
\end{equation*}
If $f$ is of class $\SC^m$, the iterated tangent map
$T^mf:T^mU\to T^m F$
is recursively defined as
\[
     T^mf:=T(T^{m-1}f).
\]
For example
\begin{equation*}
\begin{split}
     T^2f: U^2\oplus E^1\oplus E^1\oplus E^0
     \to F^2\oplus F^1\oplus F^1\oplus F^0
\end{split}
\end{equation*}
is given by
\begin{equation*}
\begin{split}
     T^2f(x,h,\hat x,\hat h)
     =\left(f(x),Df(x)h,Df(x)\hat x, D^2f(x)(h,\hat x)+Df(x)\hat h
     \right).
\end{split}
\end{equation*}

\begin{lemma}[Characterization of $\SC^2$ in terms of sc-differentials]
Assume that $f:U\to F$ is $\SCo$. Then
$f$ is $\SC^2$ if and only if the following conditions hold:
\begin{itemize}
\item[(i)]
  $f:U_2\to F_0$ is pointwise twice differentiable in the usual sense;
\item[(ii)]
  For every $x\in U_2$ the second differential $d^2f(x):E_2\oplus E_2\to F_0$
  has a continuous extension $D^2f(x):E_1\oplus E_1\to F_0$;
\item[(iii)]
  For all $k\in\N$ and $x\in U_{k+1}$, the continuous extension
  $D^2f(x):E_1\oplus E_1\to F_0$ restricts to a continuous bilinear map
  \[
     D^2f(x)|_{E_k\oplus E_k}:E_k\oplus E_k\to F_{k-1}
  \] 
  such that
  \[
     D^2f|_{U_{k+1}\oplus E_k\oplus E_k}: U_{k+1}\oplus E_k\oplus E_k\to F_{k-1}
  \]
  is continuous in the compact-open topology.
\end{itemize}
\end{lemma}

\begin{remark}[Symmetry of second scale differentials]
\label{rem:sym-D2f}
The second scale differential $D^2f(x):E_1\oplus E_1\to F_0$ is
symmetric, because the usual second differential
$d^2f(x):E_2\oplus E_2\to F_0$ is symmetric
and $E_2$ is a dense subset of the Banach space $E_1$.
\end{remark}

\begin{proposition}[Characterizing $\SC^m$ by higher
sc-differentials $D^mf(x)$]\label{prop:char-sc-diff}
Let $f:U\to F$ be $\SC^{m-1}$. Then
$f$ is $\SC^m$ iff the following conditions hold:
\begin{itemize}
\item[(i)]
  $f:U_m\to F_0$ is pointwise $m$ times differentiable in the usual sense;
\item[(ii)]
  For every $x\in U_m$ the $m^{\rm th}$ differential
  $d^mf(x):E_m\oplus\dots\oplus E_m\to F_0$
  has a continuous extension
  \[
     D^mf(x):\underbrace{E_{m-1}\oplus\dots\oplus E_{m-1}}_{\text{$m$ times}}
     \to F_0;
  \]
\item[(iii)]
  For all $k\ge m-1$ and $x\in U_{k+1}$, the continuous extension
  $D^mf(x):E_{m-1}\oplus\dots\oplus E_{m-1}\to F_0$ restricts to a
  continuous $m$-fold multilinear map

  \[
     D^mf(x):\underbrace{E_{k}\oplus\dots\oplus E_{k}}_{\text{$m$ times}}
     \to F_{k-(m-1)}=F_{k-m+1}
  \] 
  such that
  \[
     D^mf|_{U_{k+1}\oplus E_k\oplus\dots\oplus E_k}: U_{k+1}\oplus
     E_k\oplus\dots\oplus E_k\to F_{k-m+1}
  \]
  is continuous in the compact-open topology.
\end{itemize}
\end{proposition}

The higher sc-differentials $D^mf(x)$ are \emph{symmetric}
$m$-fold multilinear maps by the argument in
Remark~\ref{rem:sym-D2f}.

%%%%%%%%%%%%%%%%%%%%%%%%%%%%%%%%%%%
%%%%%%% Section: Scale smoothness %%%%%%%%
%%%%%%%%%%%%%%%%%%%%%%%%%%%%%%%%%%%
\section{Chain rule}\label{sec:chain-rule}

The following theorem was proved by Hofer, Wysocki,
and Zehnder in~\cite{Hofer:2007a}.
The proof relies heavily on the compactness condition
on the scale inclusions $E_{i+1}\INTO E_i$
in Definition~\ref{def:Banach-scale} of a Banach scale.

\begin{theorem}[Chain rule]\label{thm:chain-rule}
Consider scale Banach spaces $E,F,G$
and open subsets $U\subset E$ and $V\subset F$.
Suppose the maps $f:U\to V$ and $g:V\to G$ are of class $\SCo$.
Then the composition $g\circ f:U\to G$
is of class $\SCo$ and
\[
     T(g\circ f)=Tg\circ Tf:TU\to TG.
\]
\end{theorem}

Concerning the proof of the chain rule in~\cite[p.\,849]{Hofer:2007a}
Hofer, Wysocki, and Zehnder remark the following.

\begin{quote}
``The reader should realize that in the above proof all conditions on
$\SCo$ maps have been used, i.e. it just works.''
\end{quote}

An immediate consequence of the chain rule is the following corollary.

\begin{corollary}
Under the assumptions of Theorem~\ref{thm:chain-rule}
suppose, in addition, that $f$ and $g$ are of class $\SC^m$
where $m\in\N$. Then the composition $g\circ f:U\to G$
is of class $\SC^m$ and its $m$-fold iterated tangent map
is given by
\[
     T^m(g\circ f)=T^mg\circ T^mf:T^mU\to T^mG.
\]
\end{corollary}

%%%%%%%%%%%%%%%%%%%%%%%%%%%%%%%%%%%
%%%%%%% Section: Scale smooth actions %%%%%%
%%%%%%%%%%%%%%%%%%%%%%%%%%%%%%%%%%%
\section{Scale smooth actions}\label{sec:scale-actions}
In this section we explain that the shift map is scale smooth.
We first prove this for loop spaces $H_k=W^{k,2}(\SS^1,\R)$, however, basically the
same proof generalizes to Morse and Floer trajectory spaces
$E_k=\bigcap_{i=0}^k W^{i,p}_{\delta_k}(\R,H_{k-i})$.
Surprisingly the growth type used to define the Floer trajectory
spaces does not enter the proof.
This is due to the fact that the shift map is linear
in the second variable.

%%%%%%% Subsection: Loop spaces %%%%%%%%%%%%%%%%%%%
%%%%%%%%%%%%%%%%%%%%%%%%%%%%%%%%%%%%%%%%%%%%
\subsection*{Loop spaces}

\begin{theorem}[Shift map on loop spaces is scale smooth]
\label{thm:shift-sc-smooth}
Let $H$ be the scale Hilbert space whose levels are given by
$H_k=W^{k,2}(\SS^1,\R)$ and consider the map
\begin{equation}\label{eq:shift-map}
     \Psi:F=\R\oplus H\to H,\quad
     (\tau,v)\mapsto \tau_*v,
\end{equation}
where $\R$ carries the constant scale structure.
Then the map $\Psi$ is $\SC$-smooth.
\end{theorem}

\begin{proof}
At the point $(\tau,v)\in F_1$ the sc-differential
evaluated on $( T_1,V_1)\in F_0$ is
\[
     D\Psi(\tau,v) \bigl( T_1,V_1\bigr)=\tau_*
     v^\prime\cdot T_1+ \tau_*V_1.
\]
At the point $(\tau,v)\in F_2$ the second sc-differential $D^2\Psi(\tau,v)$
evaluated on a pair $\bigl(( T_1,V_1), ( T_2,\hat
v_2)\bigr)\in F_1$ is given by
\[
     D^2\Psi(\tau,v) \bigl(( T_1,V_1), ( T_2,V_2)\bigr)
     =\tau_* v^{\prime\prime}\cdot T_1 T_2
     +\tau_*V_1^\prime \cdot T_2
     +\tau_*V_2^\prime \cdot T_1.
\]
By induction one shows that at $(\tau,v)\in F_k$
the $k^{\rm th}$ iterated sc-differential
\[
     D^k\Psi (\tau,v):F_{k-1}\oplus\dots\oplus F_{k-1}\to H_0
\]
evaluated on $k$ elements $( T_1,V_1),\dots,( T_k,V_k)\in
F_{k-1}$ is given by the formula
\begin{equation}\label{eq:Dk-Psi}
\begin{split}
   &D^k\Psi (\tau,v)\Bigl(( T_1,V_1),\dots,( T_k,V_k)\Bigr)\\
   &=\tau_* v^{(k)}\prod_{j=1}^k T_j
     +\sum_{j=1}^k \tau_* V_j^{(k-1)}
      T_1\dots\widehat{ T_j}\dots T_k
\end{split}
\end{equation}
where the wide hat in $\widehat{ T_j}$ means to delete that term.
That the iterated sc-differentials meet the requirements of
Proposition~\ref{prop:char-sc-diff}
follows from Lemma~\ref{le:shift-C0-compopen}.
\end{proof}

%%%%%%% Subsection: Path spaces %%%%%%%%%%%%%%%%%%%
%%%%%%%%%%%%%%%%%%%%%%%%%%%%%%%%%%%%%%%%%%%%
\subsection*{Floer trajectory spaces}

\begin{theorem}[Shift map on path spaces is scale smooth]
\label{thm:main}
Let $E$ be the scale Banach space of path spaces arising in Morse
or Floer homology as introduced in
Examples~\ref{eq:path-Morse} or~\ref{eq:path-Floer}.
Then the shift map in~(\ref{eq:shift-map}) with $H$ replaced by $E$
is $\SC$-smooth.
\end{theorem}

\begin{proof}
As proof of Theorem~\ref{thm:shift-sc-smooth}.
More precisely, replace Lemma~\ref{le:shift-C0-compopen}
\begin{itemize}
\item[-]
  in the Morse case by Lemma~\ref{le:shift-C0-compopen-gen}
  with finite dimensional $H$
\item[-]
  and in the Floer case by Corollary~\ref{cor:hghjbjk}.
\end{itemize}
\end{proof}

It is surprising that the proof of
Theorem~\ref{thm:main}
can be given uniformly, independent of the
monotone unbounded function $f:\N\to(0,\infty)$.
This hinges on the observation that in formula~(\ref{eq:Dk-Psi})
the $V_j$'s only enter linearly -- there are no products.
In fact, if there would be products, the regularity
would strongly depend on the growth type of $f$,
which is well known from Sobolev theory.
It is easy to understand why the $V_j$'s
only enter linearly in the formula~(\ref{eq:Dk-Psi})
of the differential: The shift map~(\ref{eq:shift-map})
is linear in the second variable.

\begin{lemma}[Continuity in compact open topology]
\label{le:shift-C0-compopen-gen}
Let $k\in\N_0$ and pick constants $p\in(1,\infty)$ and $\delta\ge 0$.
Suppose $H$ is a separable Hilbert space,
i.e. $H$ is either isometric to $\ell^2$ or of finite dimension,
and define $W^{k,p}_\delta(\R,H)$
by~(\ref{eq:W-kp-delta}) with $\R^n$ replaced by $H$.
Then the shift map
\[
     \Psi_\tau:W^{k,p}_\delta(\R,H)\to \Psi_\tau:W^{k,p}_\delta(\R,H),\quad
     v\mapsto \tau_* v,
\]
is continuous in the compact open topology, i.e.
\[
     \lim_{\tau\to 0}\Norm{\Psi_\tau(v)-v}_{W^{k,p}_\delta(\R,H)}=0.
\]
for each $v\in W^{k,p}_\delta(\R,H)$.
\end{lemma}

\begin{proof}
Same arguments as in Lemma~\ref{le:shift-C0-compopen}.
\end{proof}

An immediate Corollary of the lemma is the following.

\begin{corollary}\label{cor:hghjbjk}
Let $k\in\N_0$ and pick constants $p\in(1,\infty)$ and $\delta_k\ge 0$.
Let $f:\N\to(0,\infty)$ be a monotone unbounded function
and consider the weighted Hilbert spaces $H_0:=\ell^2$ and
$H_j:=\ell^2_{f^j}$ for $j\in\N$ as in~(\ref{eq:ell-2-f-k}).
Then the shift map on the intersection Banach space
\[
     \Psi_\tau:
     E_k=\bigcap_{i=0}^k W^{i,p}_{\delta_k}(\R,H_{k-i})
     \to
     E_k=\bigcap_{i=0}^k W^{i,p}_{\delta_k}(\R,H_{k-i}) ,\quad
     v\mapsto \tau_* v,
\]
is continuous in the compact open topology.
\end{corollary}

%%%%%%%%%%%%%%%%%%%%%%%%%%%%%%%%%%%
%%%%%%% Section: Scale smooth actions %%%%%%
%%%%%%%%%%%%%%%%%%%%%%%%%%%%%%%%%%%
\section{Fractal Hilbert scale structures on mapping spaces}
\label{sec:mapping-spaces}
In this section we explain how fractal scale Hilbert spaces
can be used to model the targets in Floer homology.
Let $N$ be a closed manifold. Fix an integer $k_0>\frac12\dim N$
and consider the Hilbert scale defined by
\[
     \Map(N,\R^r)=\Map(N,\R^r)_0
     :=W^{k_0,2}(N,\R^r)\supset W^{k_0+1,2}(N,\R^r)\supset\dots .
\]
The spectral theory of the Laplace operator $\Delta_g$
associated to a Riemannian metric $g$ on $N$
shows that this Hilbert scale is given by the fractal
Hilbert scale $\ell^{2,f}$ associated in Example~\ref{eq:sc-ell2} to
the weight function
\begin{equation}\label{eq:dim-frac}
     f(\nu)=\nu^{2/\dim N}.
\end{equation}
Observe that $f$ only
depends on the dimension of \emph{the domain} $N$,
it is independent of the dimension of the target;
cf.~\cite{Kang:2011a}.
This phenomenon is very reminiscent of the Sobolev theory
which is sensible to the dimension of the domain, but not of the target.

%%%%%%% Subsection: Periodic boundary conditions %%%%%%%
%%%%%%%%%%%%%%%%%%%%%%%%%%%%%%%%%%%%%%%%%%%%
\subsection*{Periodic boundary conditions}

We illustrate this for the Hilbert space
$H=\Map(\SS^1,\C):=L^2(\SS^1,\C)$.
Here we get away with $k_0=0$, because we view
maps $\SS^1\to\C$ as maps $\R\to\C$ which are
$1$-periodic. So there is no need to take local
coordinate charts on $\SS^1$ and therefore
we don't need continuity of the elements of
our mapping space $\Map(\SS^1,\C)$.
The Hilbert space $H$
consists of all infinite sums of the form
\[
     v(t)=\sum_{\ell\in\Z} v_\ell e^{2\pi i\ell t}
\]
whose Fourier coefficient sequences $(v_\ell)_{\ell\in\N}\subset\C$ are square
summable, that is
\[
     \sum_{\ell\in\Z}\abs{v_\ell}^2<\infty.
\]
For $k\in\N_0$ the subspace $W^{k,2}(\SS^1,\C)$ of $L^2(\SS^1,\C)$
consists of those $v$ for which the weighted sum
$\sum_{\ell\in\Z} (1+\ell^2)^k \abs{v_\ell}^2$ is finite.
Up to equivalent weight functions
the space $W^{k,2}(\SS^1,\C)$ coincides with the $k^{\rm th}$ level
\[
     H_k:=(\ell^{2,f})_k=\ell^2_{f^k}
\]
of the scale Hilbert space $\ell^{2,f}$ associated in
Example~\ref{eq:sc-ell2} to the weight function
$f:\N\to(0,\infty)$, $\nu\mapsto \nu^2$.
This is consistent with formula~(\ref{eq:dim-frac}),
because $\dim\SS^1=1$.

%%%%%%% Subsection: Lagrangian boundary conditions %%%%%%%
%%%%%%%%%%%%%%%%%%%%%%%%%%%%%%%%%%%%%%%%%%%%
\subsection*{Lagrangian boundary conditions}

Consider the relative mapping space
\[
     H=\Map([0,1],\{0,1\};\C,\R):=W^{1,2}(([0,1],\{0,1\}),(\C,\R))
\]
that consists of all $W^{1,2}$ paths $\gamma:[0,1]\to\C=\R\times i\R$ whose initial
and end points lie on the real line $\R$.
We define a scale structure on $H$ by choosing as $k^{\rm th}$ level
the space
\[
     H_k:=\left\{
     \gamma\in W^{k+1,2}([0,1],\C)\,\big|\,
     \gamma^{(\ell)}(0), \gamma^{(\ell)}(1)\in i^\ell\R,\, 0\le\ell\le k
     \right\}.
\]
Note that we not have just Lagrangian boundary conditions
for the path, but also for its derivatives.
This is crucial to get a fractal scale Hilbert structure.
These boundary conditions also appeared in the
thesis of Sim\v{c}evi\'{c}~\cite{Simcevic:2014a}
in her Hardy space approach to gluing.
Observe that these boundary conditions are well posed,
because the $k^{\rm th}$ derivative of such a path is still
continuous by the Sobolev embedding theorem $W^{k+1,2}\INTO C^k$
on the 1-dimensional domain $[0,1]$.

Given a path $\gamma\in H_k$, consider the associated loop in $\C$
defined by doubling
\[
     \Gamma_\gamma (t):=\begin{cases}
       \gamma(2t)&\text{, $t\in[0,\frac12]$,}\\
       \widebar\gamma(2-2t)&\text{, $t\in[\frac12,1]$,}\\
     \end{cases}
\]
\begin{figure}%[h]
  \centering
  \includegraphics%[width=0.9\textwidth]
                             [height=4.5cm]
                             {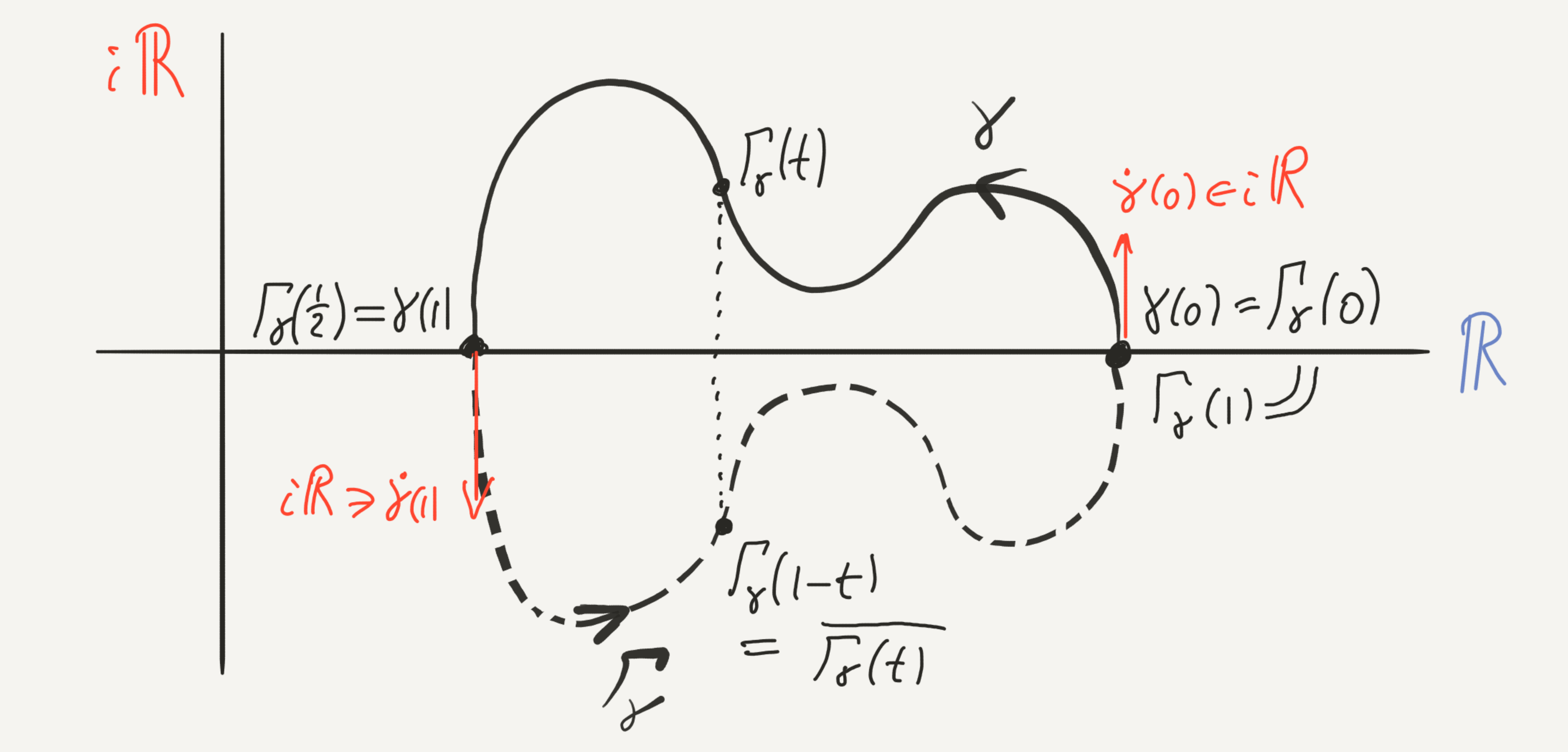}
  \caption{Doubling a path $\gamma\in H_k$ gives a loop $\Gamma_\gamma$}
  \label{fig:fig-LAG-doubling-map}
\end{figure}
where $\bar z=x-iy$ denotes complex conjugation of a complex number
$z=x+iy$. Note that $\Gamma_\gamma$ is indeed a loop
\[
     \Gamma_\gamma (1)=\widebar\gamma(0)=\gamma(0)=\Gamma_\gamma(0)
\]
where the second step holds due to the condition that initial 
points of our paths lie on the real line $\R\subset\C$.
The real endpoint condition guarantees that
the loop is also continuous at $t=\frac12$, hence everywhere.
We claim that
\[
     \Gamma_\gamma \in W^{k+1,2}(\SS^1,\C).
\]
Because $\gamma\in W^{k+1,2}([0,1],\C)$, it suffices to show that
$\Gamma_\gamma$ is $k$ times differentiable at the points $0$ and
$\frac12$. This follows from the boundary conditions imposed in the
definition of the space $H_k$.

As illustrated by Figure~\ref{fig:fig-LAG-doubling-map}
The doubling map $\gamma\mapsto\Gamma_\gamma$ gives us an embedding
\[
     I:H_k\INTO W^{k+1,2}(\SS^1,\C),\quad
     \gamma\mapsto\Gamma_\gamma .
\]
The elements of the image of $I$ are precisely
those $W^{k+1,2}$ loops $\Gamma$ in $\C$ that
are symmetric with respect to the real line $\R\subset\C$,
more precisely
\begin{equation}\label{eq:Gamma-symm}
     \Gamma(t)=\overline{\Gamma(1-t)},\quad t\in[0,1].
\end{equation}
Indeed suppose $\Gamma\in W^{k+1,2}(\SS^1,\C)$
satisfies~(\ref{eq:Gamma-symm}), taking its first half
$\gamma_\Gamma (t):=\Gamma(t/2)$ for $t\in[0,1]$
it follows from~(\ref{eq:Gamma-symm}) that $\gamma_\Gamma \in H_k$ and
$I(\gamma_\Gamma )=\Gamma_{\gamma_\Gamma}=\Gamma$.

Suppose $\Gamma\in\im(I)$ lies in the image of $I$, that is
$\Gamma:\SS^1\to\C$ is of class $W^{k+1,2}$ and
satisfies~(\ref{eq:Gamma-symm}).
Writing $\Gamma$ as a Fourier series we obtain that
\[
     \sum_{\ell\in\Z} v_k e^{2\pi i\ell t}
     =\Gamma(t)=\overline{\Gamma(1-t)}=\overline{\Gamma(-t)}
     =\overline{\sum_{\ell\in\Z} v_k e^{-2\pi i\ell t}}
     =\sum_{\ell\in\Z} \widebar v_k e^{2\pi i\ell t}.
\]
This shows that all Fourier coefficients $v_k=\widebar v_k$
are real. In particular, up to scale isomorphism,
the scale relative mapping space $H$
is scale isomorphic to the fractal Hilbert scale $\ell^{2,f}$
for the weight function $f(\nu)=\nu^2$,
in symbols
\[
     H=W^{1,2}(([0,1],\{0,1\}),(\C,\R))\simeq \ell^{2,f},\quad f(\nu)=\nu^2 .
\]

%%%%%%% Subsection: Growth type  %%%%%%%%%%%%%%%%%%%
%%%%%%%%%%%%%%%%%%%%%%%%%%%%%%%%%%%%%%%%%%%%
\subsection*{The growth type of various Floer homologies}

In view of the discussion before we can now list the
growth type of the various Floer homologies mentioned in the
introduction.

The growth types are different, however, we point out that
the main result, Theorem~\ref{thm:main},
is independent of the growth type and therefore applies
to all of the following.

\vspace{.5cm}
\begin{tabular}{llll}
%\begin{tabular}{llr}
\toprule
%\multicolumn{2}{c}{Item} \\
%\cmidrule(r){1-2}
     Floer homology
  & Order
  & Mapping space
  & Growth type
\\
%%%%%%%%%%%%%%%%%%%%%%%%%%%%
\midrule
     Periodic
  & $1^{\rm st}$ %order
  & loop space
  & $f(\nu)=\nu^2$
\\
%%%
     Lagrangian
  & $1^{\rm st}$ %order
  & path space
  & $f(\nu)=\nu^2$
\\
%%%
     Hyperk\"ahler
  & $1^{\rm st}$ %order
  & $\Map(M^3,\R^{2n})$
  & $f(\nu)=\nu^{2/3}$
\\
%%%
     Heat flow
  & $2^{\rm nd}$ %order
  & loop space
  & $f(\nu)=\nu^4$
\\
%%%%%%%%%%%%%%%%%%%%%%%%%%%%
\bottomrule
\end{tabular}

%%%%%%%%%%%%%%%%%%%%%%%%%%%%%%%%%%%
%%%%%%% Section: Scale smooth actions %%%%%%
%%%%%%%%%%%%%%%%%%%%%%%%%%%%%%%%%%%
\section{Banach scale structures -- main examples}
\label{sec:Bsc-str-ex}
In this section we show that the examples of scale structures
introduced in Section~\ref{sec:scale-structures}
actually satisfy the axioms of scale structures.

%\newpage
%%%%%%% Subsection: Lagrangian boundary conditions %%%%%%%
%%%%%%%%%%%%%%%%%%%%%%%%%%%%%%%%%%%%%%%%%%%%
\subsection*{Fractal scale Hilbert spaces}

Consider a monotone unbounded function $f:\N\to(0,\infty)$
and consider the weighted Hilbert spaces $H_0:=\ell^2$ and
$H_k:=\ell^2_{f^k}$ for $k\in\N$ as in Example~\ref{eq:sc-ell2}.
Our aim is to show that the nested sequence of Hilbert spaces
$\ell^{2,f}=\ell^2\supset\ell^2_f\supset\ell^2_{f^2}\dots$ carries
the structure of a scale Hilbert space, that is
compact inclusions and density of $\cap_{k=0}^\infty \ell^2_{f^k2}$
in each level $\ell^2_{f^k2}$.

\textbf{Compact inclusions.}
Consider the finite dimensional subspace
$V_N:=\{\sum_{i=1}^N a_i e_i\mid a_i\in\R\}\subset\ell^2_f$,
the orthogonal projection $\pi_N:\ell^2_f\to V_N$, and the
non-commutative diagram
\begin{equation*}
\begin{tikzcd} %[row sep=tiny]
\ell^2_f
\arrow[rr, hook, "I"]
\arrow[d, "\pi_N"']
  && \ell^2
\\
V_N
\arrow[rru, hook, "I_N"']
  &&
\end{tikzcd}
\end{equation*}
By finite dimension of $V_N$ the inclusion $I_N$ is
a compact operator. Therefore the composition
$I^N:=I_N\pi_N:\ell^2_f\to\ell^2$ is compact.
The estimate
\begin{equation*}
\begin{split}
     \Norm{I-I^N}_{\Ll(\ell^2_f,\ell^2}
   &=\sup_{\Norm{v}_{\ell^2_f}=1}\Norm{(I-I^N)v}_{\ell^2}\\
   &=\sup_{\Norm{v}_{\ell^2_f}=1}\Norm{(\Id-\pi_N)v}_{\ell^2}\\
   &\le 1/f(N)
\end{split}
\end{equation*}
shows that $I^N\to I$, as $N\to\infty$, in the operator norm topology.
Hence $I$ is compact by Theorem~\ref{thm:compact-limit-op}.

\textbf{Density.}
Let $V=\cup_{N=1}^\infty V_N$ be the union of all the $V_N$.
The inclusions
\[
     V
     \subset
     \bigcap_{k=0}^\infty \ell^2_{f^k}
     \subset \ell^2_{f^k}
\]
together with density of $V$ in $\ell^2_{f^k}$
implies density of $\cap_{k=0}^\infty \ell^2_{f^k}$ in each weighted
Hilbert space $\ell^2_{f^k}$. We proved the following theorem.
 
\begin{theorem}[The fractal Hilbert scale]\label{thm:fractal-scale}
The sequence of fractal Hilbert spaces $H_k=\ell^2_{f^k}$
defined by~(\ref{eq:ell-2-f-k})
forms a Banach scale.
\end{theorem}

%%%%%%% Subsection: Lagrangian boundary conditions %%%%%%%
%%%%%%%%%%%%%%%%%%%%%%%%%%%%%%%%%%%%%%%%%%%%
\subsection*{Morse path spaces}

Our aim is to show that the sequence of Morse path spaces
$E_k=W^{k,p}_{\delta_k}(\R,\R^n)$ introduced in
Example~\ref{eq:path-Morse} has the two defining properties of
a Banach scale, compact inclusions and density.

\begin{theorem}[The Morse path Banach scale]\label{thm:Morse-path}
The sequence of Morse path spaces $E_k=W^{k,p}_{\delta_k}(\R,\R^n)$
defined by~(\ref{eq:Morse-path-Ek})
forms a Banach scale.
\end{theorem}

\begin{proof}
Density: The inclusions
\[
     C^\infty_{\rm c}(\R,\R^n)
     \subset
     E_\infty:=\bigcap_{k=0}^\infty W^{k,p}_{\delta_k}(\R,\R^n)
     \subset W^{k,p}_{\delta_k}(\R,\R^n) =:E_k
\]
together with density of the set of compactly supported smooth functions
in the Banach space $W^{k,p}_{\delta_k}(\R,\R^n)$
implies density of $E_\infty$ in each level $E_k$.

Compact inclusions: Proposition~\ref{prop:Morse-path-compact-k}.
\end{proof}

\begin{proposition}[Compact inclusions]\label{prop:Morse-path-compact-k}
Suppose $k\in\N$ and $p\in(1,\infty)$. For non-negative reals
$\delta_1>\delta_0$ the inclusion
\[
     I:W^{k,p}_{\delta_1}(\R)\INTO W^{k-1,p}_{\delta_0}(\R)
\]
is a compact linear operator. The Banach spaces are
defined by~(\ref{eq:W-kp-delta}).
\end{proposition}

In order to prove the proposition we first prove two lemmas.

\begin{lemma}[$k=1$, $\delta_0=0$]\label{le:Morse-path-compact}
For constants $p\in(1,\infty)$ and $\delta>0$ the inclusion
\[
     I:W^{1,p}_\delta(\R)\INTO L^p(\R)
\]
is compact where the Banach space $W^{1,p}_\delta$ is defined
by~(\ref{eq:W-kp-delta}).
\end{lemma}

Without the exponential weights the inclusion
$W^{1,p}(\R)\INTO L^p(\R)$
is in general not compact, as shown by the sequence
$v_\nu(t):=v(t-\nu)$ of right shifts of a given function $v$
of positive norm, for instance a bump function $v$.

\begin{proof}
For $T>0$ consider the continuous operators given by restriction
\[
     R_T: W^{1,p}_\delta(\R)\to W^{1,p}([-T,T]),\quad
     v\mapsto v|_{[-T,T]},
\]
and extension by zero $E_T: L^{p}([-T,T])\to L^p(\R)$.
Since $p>\dim [-T,T]=1$ the inclusion operator
$
     I_T: W^{1,p}([-T,T])\INTO L^p([-T,T])
$
is compact by the Sobolev embedding theorem.
Hence the composition
\[
     I^T:=E_T I_T R_T:W^{1,p}_\delta(\R)\to L^p(\R)
\]
is compact. Because the set of compact linear operators
is norm closed, see Theorem~\ref{thm:compact-limit-op},
it suffices to show that $I^T\to I$ in the
norm topology, as $T\to\infty$. Indeed
\begin{equation*}
\begin{split}
     \Norm{I-I^T}_{\Ll (W^{1,p}_\delta , L^p)}
   &=\sup_{\Norm{v}_{W^{1,p}_\delta}=1} \Norm{(I-I^T)v}_{L^p}\\
   &=\sup_{\Norm{v}_{W^{1,p}_\delta}=1}\Norm{v|_{(-\infty,-T]\cup[T,\infty)}}_{L^p}\\
   &\le e^{-\delta T}
\end{split}
\end{equation*}
whenever $T\ge 1$.
To see the final estimate observe that
\begin{equation*}
\begin{split}
     \Norm{v|_{(-\infty,-T]\cup[T,\infty)}}_{L^p}
   &\le\frac{1}{e^{\delta T}}\Norm{\gamma_\delta\cdot v|_{(-\infty,-T]\cup[T,\infty)}}_{L^p}\\
   &\le\frac{1}{e^{\delta T}}\Norm{v|_{(-\infty,-T]\cup[T,\infty)}}_{W^{1,p}_\delta}\\
   &\le\frac{1}{e^{\delta T}}\Norm{v}_{W^{1,p}_\delta}\\
   &=e^{-\delta T}
\end{split}
\end{equation*}
whenever $\norm{v}_{W^{1,p}_\delta}=1$ and
where step one uses that $T\ge 1$.
\end{proof}

\begin{lemma}[General $k$, $\delta_0=0$]
\label{le:Morse-path-compact-k}
Given $k\in\N$ and reals $p\in(1,\infty)$ and $\delta>0$, then the
inclusion
\[
     I:W^{k,p}_\delta(\R)\INTO W^{k-1,p}(\R)
\]
is compact. The Banach spaces $W^{k,p}_\delta$ are defined by~(\ref{eq:W-kp-delta}).
\end{lemma}

\begin{proof}
The lemma follows by induction on $k$. For $k=1$ the assertion is true
by Lemma~\ref{le:Morse-path-compact}. To prove the induction step
$k\Rightarrow k+1$ suppose the lemma holds true for $k$.
Let $v_\nu$ be a sequence in the unit ball of $W^{k+1,p}_\delta(\R)$.
Hence both $v_\nu$ and its derivative $\dot v_\nu$ lie in the unit
ball of $W^{k,p}_\delta(\R)$. By induction hypothesis $k$
there exist elements $v,w\in W^{k-1,p}_\delta(\R)$
and a subsequence, still denoted by $v_\nu$, such that
as $\nu\to\infty$ it holds that
\[
     \text{$v_\nu\stackrel{W^{k-1,p}}{\longrightarrow} v$
     \quad and \quad
     $\dot v_\nu\stackrel{W^{k-1,p}}{\longrightarrow} w$.}
\]
Note that $w$ is equal to the weak derivative $\dot v$. Indeed the
definition of the weak derivative provides the first of the identities
\begin{equation}\label{eq:w=dot-v}
\begin{split}
     \int_\R\phi\dot v
     =-\int_\R\dot\phi  v
     =-\lim_{\nu\to\infty} \int_\R\dot\phi  v_\nu
     =\lim_{\nu\to\infty} \int_\R\phi \dot v_\nu
     =\int_\R\phi w
\end{split}
\end{equation}
which hold true for every test function $\phi\in C^\infty_0(\R)$.
In particular, the element $v$ lies in $W^{k,p}$
and $v_\nu\to v$ in $W^{k,p}$.
\end{proof}

\begin{proof}[Proof of Proposition~\ref{prop:Morse-path-compact-k}]
Consider the isomorphisms
\[
     T:W^{k,p}_{\delta_1}\to W^{k,p}_{\delta_1-\delta_0}
     ,\quad
     S:W^{k-1,p}_{\delta_0}\to W^{k-1,p}
     ,\quad v\mapsto \gamma_{\delta_0}\cdot v ,
\]
which both act by multiplication by the weight function $\gamma_{\delta_0}$.
The assertion of the proposition -- compactness of the
inclusion $I:W^{k,p}_{\delta_1}(\R)\INTO W^{k-1,p}_{\delta_0}(\R)$ --
follows since the diagram commutes
\begin{equation}\label{eq:morse-coboundary}
\begin{tikzcd} [column sep=tiny]
W^{k,p}_{\delta_1-\delta_0}
\arrow[rr, hook]
  && W^{k-1,p}
\\
W^{k,p}_{\delta_1}
\arrow[rr, hook, "I"]
\arrow[u, "T", "\simeq"']
  && W^{k-1,p}_{\delta_0}
    \arrow[u, "\simeq", "S"']
\end{tikzcd}
\end{equation}
and the upper inclusion is compact by
Lemma~\ref{le:Morse-path-compact-k}.
\end{proof}

%\newpage
%%%%%%% Subsection: Lagrangian boundary conditions %%%%%%%
%%%%%%%%%%%%%%%%%%%%%%%%%%%%%%%%%%%%%%%%%%%%
\subsection*{Floer path spaces}

Our aim is to show that the sequence of Floer path spaces $E_k$ introduced in
Example~\ref{eq:path-Floer} has the two defining properties of
a Banach scale, compact inclusions and density.
Hence from now on let $f:\N\to(0,\infty)$ be a monotone unbounded function
and consider the weighted Hilbert spaces $H_0:=\ell^2$ and
$H_k:=\ell^2_{f^k}$ for $k\in\N$ as in~(\ref{eq:ell-2-f-k}).

\begin{theorem}[The Floer path Banach scale]\label{thm:Floer-path}
The sequence of Floer path spaces $E_k$
defined by~(\ref{eq:E_k-Floer}) forms a Banach scale.
\end{theorem}

\begin{proof}
Density: Consider the dense subset $V$ of $\ell^2$
that consists of all finite sums
\[
     V:=\biggl\{\text{$\sum_{i=1}^N a_i e_i$ $\Big|$
     $N\in\N$, $a_1,\dots,a_N\in\R$}\biggr\}
     \subset \ell^2.
\]
The inclusions
$
     V
     \subset
     \bigcap_{k=0}^\infty \ell^2_{f^k}
     \subset \ell^2_{f^k}
$
together with density of $V$ in $\ell^2_{f^k}$
implies density of $V$ in each weighted Hilbert space $\ell^2_{f^k}$.
The inclusions
\[
     C^\infty_{\rm c}(\R,V)
     \subset
     E_\infty:=\bigcap_{k=0}^\infty E_k
     \subset E_k
\]
together with density of the set $C^\infty_{\rm c}(\R,V)$
in the Banach space $E_k$
implies density of $E_\infty$ in each level $E_k$.

Compact inclusions: Proposition~\ref{prop:Floer-path-compact-k}.
\end{proof}

\begin{proposition}[Compact inclusions]\label{prop:Floer-path-compact-k}
Suppose $k\in\N$ and $p\in(1,\infty)$. For non-negative reals
$\delta_k>\delta_{k-1}$ the inclusion
\[
     I:E_k=\bigcap_{i=0}^k W^{i,p}_{\delta_k}(\R,H_{k-i})
     \to
     \bigcap_{i=0}^{k-1} W^{i,p}_{\delta_{k-1}}(\R,H_{k-1-i})=E_{k-1}
\]
is a compact linear operator.
\end{proposition}

In order to prove the proposition we first prove two lemmas.

\begin{lemma}[$k=1$, $\delta_0=0$]\label{le:Floer-path-comp-k=1}
Pick reals $p>1$ and $\delta>0=\delta_0$. Then the inclusion
\[
     I:E_1=W^{1,p}_\delta(\R,H_0)\cap L^p _\delta(\R,H_1)
     \to
     L^p(\R,H_0)=E_0
\]
is a compact linear operator.
\end{lemma}

Recall from Example~\ref{eq:path-Floer} that the norm on an
intersection of Banach spaces is the maximum of the individual norms.

\begin{proof}[Proof of Lemma~\ref{le:Floer-path-comp-k=1} ($k=1$)]
Denote by $e_i=(0,\dots,0,1,0,\dots)$ the sequence whose members are
all $0$ except for member $i$ which is $1$.
The set of all $e_i$ not only form an orthonormal basis of the Hilbert
space $H_0=\ell^2$, but simultaneously an orthogonal basis of
$H_1=\ell^2_f$, although not of unit length any more.

For $N\in\N$ consider the subspace
$V_N=\SPAN\{e_1,\dots,e_N\}\subset H_0$ of finite dimension
and the corresponding orthogonal projection $\pi_N:H_0=\ell^2\to V_N$.
Its restriction $\pi_N|_{\ell^2_f}:H_1=\ell^2_f\to V_N$
is also an orthogonal projection.

Define a linear projection by
\begin{equation*}
\begin{split}
     \Pi_N:W^{1,p}_\delta(\R,H_0)\cap L^p_\delta(\R,H_1)
   &\to W^{1,p}_\delta(\R,V_N)
     \\
     v
   &\mapsto \pi_N\circ v
\end{split}
\end{equation*}
The linear operator given by inclusion and denoted by
\[
     I_N: W^{1,p}_\delta(\R,V_N)\INTO L^p(\R,V_N),\quad
     v\mapsto v,
\]
is compact by Lemma~\ref{le:Morse-path-compact} for $I=I_N$.
The inclusion $j_N:V_N\INTO \ell^2=H_0$ induces
the inclusion
\[
     J_N: L^p(\R,V_N)\INTO L^p(\R,H_0),\quad
     v\mapsto j_N\circ v.
\]
The inclusion given by composition of bounded linear operators
\[
     I^N:=J_N I_N \Pi_N:E_1=W^{1,p}_\delta(\R,H_0)\cap L^p_\delta(\R,H_1)
     \to
     L^p(\R,H_0)=E_0
\]
is compact since $I_N$ is compact.

To see that $I^N$ converges to $I$ in the norm topology
observe that
\[
     \Norm{I-I^N}_{\Ll(E_1,E_0)}
     =\sup_{\Norm{v}_{E_1}=1} \Norm{(I-I^N)v}_{E_0}
     =\sup_{\Norm{v}_{E_1}=1} \Norm{(\Id-\pi_N)v}_{L^p(\R,H_0)}.
\]
Observe that
\begin{equation*}
\begin{split}
    \Norm{(\Id-\pi_N)v}_{L^p(\R,H_0)}
   &=\biggl(\int_{-\infty}^\infty 
     \underbrace{\Norm{(\Id-\pi_N)v(s)}_{H_0}^p}
     _{\le\frac{1}{f(N)^p}\norm{v(s)}_{H_1}^p}
     \, ds
     \biggr)^{1/p}\\
   &\le\frac{1}{f(N)}
     \Norm{v}_{L^p_\delta(\R,H_1)}\\
   &\le\frac{1}{f(N)}
     \Norm{v}_{W^{1,p}_\delta(\R,H_0)\cap L^p_\delta(\R,H_1)}.
\end{split}
\end{equation*}
This proves that $\norm{I-I^N}_{\Ll(E_1,E_0)}\le 1/f(N)$.
By unboundedness of $f$ the sequence of compact
operators $I^N$ converges to $I$ in norm. Thus the limit $I$ is
compact, too, by Theorem~\ref{thm:compact-limit-op}.
\end{proof}

In view of the fractal structure, i.e. all levels are self-similar
(isometrically isomorphic), an immediate Corollary
to Lemma~\ref{le:Floer-path-comp-k=1} is the following.

\begin{corollary}\label{cor:bjbnjgft}
Pick reals $p>1$ and $\delta>0=\delta_0$. Then each of the
inclusions
\[
     I:W^{1,p}_\delta(\R,H_k)\cap L^p_\delta(\R,H_{k+1})
     \to
     L^p(\R,H_k),\quad k\in\N_0,
\]
is a compact linear operator.
\end{corollary}

\begin{lemma}[$k\in\N$, $\delta_0=0$]\label{le:Floer-path-comp-k}
For $p\in(1,\infty)$ and $\delta>0$ the inclusion
\[
     I:\bigcap_{i=0}^k W^{i,p}_{\delta}(\R,H_{k-i})
     \to
     \bigcap_{i=0}^{k-1} W^{i,p}(\R,H_{k-1-i})
\]
is a compact linear operator. The Banach spaces $W^{i,p}_{\delta}$ are
defined by~(\ref{eq:W-kp-delta}).
\end{lemma}

\begin{proof}
Lemma~\ref{le:Floer-path-comp-k}
follows from Lemma~\ref{le:Floer-path-comp-k=1} ($k=1$)
by induction similarly as in the Morse case
(where Lemma~\ref{le:Morse-path-compact-k}
followed from Lemma~\ref{le:Morse-path-compact}).
In order to illustrate the adjustments one has to do,
we show how the case $k=2$ follows from the case $k=1$
which is Lemma~\ref{le:Floer-path-comp-k=1}.

Case $k=1$ $\Rightarrow$ case $k=2$:
Pick a sequence $v_\nu$ in the unit ball of the space
\[
     W^{2,p}_\delta(\R,H_0)\cap
     W^{1,p}_\delta(\R,H_1)\cap
     L^p_\delta(\R,H_2).
\]
So $v_\nu$ is a sequence in the unit ball of
\[
     W^{1,p}_\delta(\R,H_1)\cap
     L^p_\delta(\R,H_2).
\]
Hence by Corollary~\ref{cor:bjbnjgft} for $k=1$
there is a subsequence, still denoted by $v_\nu$,
and an element $v\in L^p(\R,H_1)$ such that
\[
     \text{$v_\nu\longrightarrow v$\quad in $L^p(\R,H_1)$}.
\]
Moreover, the weak derivatives $\dot v_\nu$ form a sequence
in the unit ball of
\[
     W^{1,p}_\delta(\R,H_0)\cap
     L^p_\delta(\R,H_1).
\]
Hence by Lemma~\ref{le:Floer-path-comp-k=1}
there is a subsequence, still denoted by $v_\nu$,
and an element $w\in L^p(\R,H_0)$ such that
\[
     \text{$\dot v_\nu\longrightarrow w$\quad in $L^p(\R,H_0)$}.
\]
Similarly as in~(\ref{eq:w=dot-v}) one gets $w=\dot v$.
Hence $v$ is in $W^{1,p}(\R,H_0)\cap L^p(\R,H_1)$
and $v_\nu\to v$ in  $W^{1,p}(\R,H_0)\cap L^p(\R,H_1)$.
This shows that the inclusion
\[
     W^{2,p}_\delta(\R,H_0)\cap
     W^{1,p}_\delta(\R,H_1)\cap
     L^p_\delta(\R,H_2)
     \INTO
     W^{1,p}(\R,H_1)\cap
     L^p(\R,H_2)
\]
is a compact linear operator.

Case $k\Rightarrow k+1$: Follows along similar lines as
$k=1\Rightarrow k=2$.
\end{proof}

\begin{proof}[Proof of Proposition~\ref{prop:Floer-path-compact-k}]
As in the Morse case, Lemma~\ref{le:Floer-path-comp-k}
implies Proposition~\ref{prop:Morse-path-compact-k}
in view of the commutative diagram~(\ref{eq:morse-coboundary}).
\end{proof}

%%%%%%%%%%%%%%%%%%%%%%%%%%%%%%%%%%%
%%%%%%%%%%%%%%%%%%%%%%%%%%%%%%%%%%%
%%%%%%% Appendix:  %%%%%%%%%%%%%%%%%%%
%%%%%%%%%%%%%%%%%%%%%%%%%%%%%%%%%%%
%%%%%%%%%%%%%%%%%%%%%%%%%%%%%%%%%%%
\appendix

%%%%%%%%%%%%%%%%%%%%%%%%%%%%%%%%%%%
%%%%%%% Functional Analysis:  %%%%%%%%%%%%%
%%%%%%%%%%%%%%%%%%%%%%%%%%%%%%%%%%%
\section{Background from Functional Analysis}

%%%%%%%%%%%%%%%%%%%%%%%%%%%%%%%%%%%
%%%%%%% Section: Scale smooth actions %%%%%%
%%%%%%%%%%%%%%%%%%%%%%%%%%%%%%%%%%%
\subsection{Compact operators}
A useful fact to prove compactness of a linear operator
is that the space of compact operators is closed
in the space of bounded linear operators with respect
to the operator norm topology. We use this fact heavily
in Section~\ref{sec:Bsc-str-ex}. For the readers convenience
we recall in this section the proof of this well knwon fact.

Suppose $E$ and $F$ are Banach spaces.
Let $\Ll(E,F)$ be the Banach space of bounded
linear operators $T:E\to F$ whose operator norm defined by
\[
     \Norm{T}=\Norm{T}_{\Ll(E,F)}:=\sup_{\Norm{x}_E=1}\Norm{Tx}_F
\]
is finite. An operator $T\in\Ll(E,F)$ is called \textbf{compact}
if the image under $T$ of any bounded sequence $x_\nu\in E$
admits a convergent subsequence. Since $T$ is linear
it suffices to show this for sequences in the unit ball of $E$.

\begin{theorem}\label{thm:compact-limit-op}
Let $T_\nu\in\Ll(E,F)$ be a sequence of compact linear operators
which converges in the operator topology to $T\in\Ll(E,F)$.
Then $T$ is compact.
\end{theorem}

\begin{proof}
Let $x_k\in E$ be a sequence in the unit ball of $E$. 
Because each $T_\nu$ is compact, by a diagonal argument there is a
subsequence $x_{k_j}$ such that each image sequence $(T_\nu
x_{k_j})_j$ converges in $F$.

We claim that $(T x_{k_j})_j$ is a Cauchy sequence in $F$.
In order to see this pick $\eps>0$.
Because $T_\nu\to T$ in the operator topology, there is $\nu_0\in\N$
such that $\norm{T-T_{\nu_0}}\le\eps/3$.
Because the sequence $(T_{\nu_0} x_{k_j})_j$ converges in $F$,
it is a Cauchy sequence. In particular, there is $j_0\in\N$
such that for every $j_1,j_2\ge j_0$ we have
\[
    \Norm{T_{\nu_0} x_{k_{j_1}}-T_{\nu_0} x_{k_{j_2}}}_F\le\eps/3.
\]
We estimate
\begin{equation*}
\begin{split}
   &\Norm{T x_{k_{j_1}}-T x_{k_{j_2}}}_F\\
   &\le\Norm{T x_{k_{j_1}}-T_{\nu_0} x_{k_{j_1}}}_F
     +\Norm{T_{\nu_0} x_{k_{j_1}}-T_{\nu_0} x_{k_{j_2}}}_F
     +\Norm{T_{\nu_0} x_{k_{j_2}}-T x_{k_{j_2}}}_F\\
   &\le\Norm{T -T_{\nu_0}}\cdot\Norm{x_{k_{j_1}}}_E
     +\Norm{T_{\nu_0} x_{k_{j_1}}-T_{\nu_0} x_{k_{j_2}}}_F
     +\Norm{T -T_{\nu_0}}\cdot\Norm{x_{k_{j_2}}}_E\\
   &\le \eps/3\cdot 1+\eps/3+\eps/3\cdot 1
     =\eps.
\end{split}
\end{equation*}
This shows that $T x_{k_{j}}$ is a Cauchy sequence in $F$.
Since $F$ is a Banach space each Cauchy sequence converges.
Hence the linear operator $T$ is compact.
\end{proof}

%%%%%%%%%%%%%%%%%%%%%%%%%%
%%%%%%% Section: Scale smooth actions %%%%%%
%%%%%%%%%%%%%%%%%%%%%%%%%%%%%%%%%%%
\subsection{Hilbert space valued Sobolev theory}
\label{sec:Hilb-Sob}

In this appendix we recall Sobolev theory for Hilbert space
valued functions in the separable case. For a more general
treatment, which as well treats non-separable Banach spaces,
see e.g.~\cite{Cohn:1993a,evans:1998a,Kreuter:2015a,Yosida:1995a}.
In particular, we recall that Hilbert valued Sobolev spaces
are complete and therefore Banach spaces.

Suppose $H$ is a separable Hilbert space, i.e. $H$ is
isometrically isomorphic to $\ell^2$ or of finite dimension,
with inner product $\INNER{\cdot}{\cdot}$ and induced norm
$\norm{\cdot}$. Throughout $H$ is endowed with the 
\textbf{\boldmath Borel $\sigma$-algebra} $\Bb=\Bb(H)$,
i.e. the smallest $\sigma$-algebra that contains the open sets.
Let $\Bb=\Bb(\R)$ be the Borel $\sigma$-algebra
on the real line and $\Aa=\Aa(\R)$ be the Lebesgue $\sigma$-algebra.
The elements of a $\sigma$-algebra are called \textbf{measurable sets}.
Recall that a map is called measurable
if pre-images of measurable sets are measurable.

%%%%%%% Subsection: Lagrangian boundary conditions %%%%%%%
%%%%%%%%%%%%%%%%%%%%%%%%%%%%%%%%%%%%%%%%%%%%
\subsection*{The Banach space \boldmath$L^1(\R,H)$}

We need the following theorem of Pettis
which makes use of the fact that our
Hilbert space is separable.

\begin{theorem}[{Pettis~\cite{Pettis:1938a}}]\label{thm:Pettis}
Consider a Hilbert space valued function $f:\R\to H$.
The following assertions are equivalent.
\begin{itemize}
\item[1)]
  Every function $\inner{f}{x} :(\R,\Aa)\to(\R,\Bb)$, where $x\in H$,
  is measurable.
\item[2)]
  The map $f:(\R,\Aa)\to(H,\Bb)$ is measurable.
\end{itemize}
\end{theorem}

\begin{remark}\label{rem:hghgj-1}
That 2) implies 1) follows from
two facts in measure theory. Firstly,
continuous maps are measurable and, secondly,
compositions of measurable maps are measurable.
\end{remark}

\begin{remark}\label{rem:hghgj-2}
By the same reasoning as in Remark~\ref{rem:hghgj-1}
if $f:\R\to H$ meets one, hence both, conditions
in Theorem~\ref{thm:Pettis} it follows that the map
$\norm{f}:(\R,\Aa)\to(\R,\Bb)$ is measurable.
\end{remark}

\begin{definition}\label{def:integrable}
A Hilbert space valued function $f:\R\to H$ is called
\textbf{integrable} if it satisfies the following two conditions.
\begin{itemize}
\item[(i)]
  Every function $\inner{f}{x} :(\R,\Aa)\to(\R,\Bb)$, where $x\in H$,
  is measurable.
\item[(ii)]
  The integral $\int_\R\norm{f(t)} dt<\infty$ is finite.
\end{itemize}
By $\Ll^1(\R,H)$ we denote the set  of integrable Hilbert space
valued functions.
\end{definition}

\begin{proposition}\label{prop:vector-space}
The set $\Ll^1(\R,H)$ is a real vector space.
\end{proposition}

\begin{proof}
Suppose $f,g\in\Ll^1(\R,H)$ and $\lambda,\mu\in\R$
we need to show that $\lambda f+\mu g$ satisfies (i) and (ii)
in the definition.
Part (i) follows from the fact that
the space of measurable functions $(\R,\Aa)\to(\R,\Bb)$ is a vector space.
To prove part (ii) we observe that by part (i)
combined with Pettis' Theorem~\ref{thm:Pettis}, 1)$\Rightarrow$
2),
and Remark~\ref{rem:hghgj-2} the map
$
     \norm{\lambda f+\mu g}: (\R,\Aa)\to(\R,\Bb)
$
is measurable. Therefore the integral
\[
     \int_\R \Norm{\lambda f+\mu g}
     \le \Abs{\lambda}\int_\R \Norm{f}
     +\Abs{\mu}\int_\R\Norm{g}
     <\infty
\]
is finite. This proves part (ii), hence the proposition.
\end{proof}

\begin{proposition}\label{prop:Ll(R,H)-complete}
The vector space $\Ll^1(\R,H)$ is complete with respect
to the semi-norm defined by
\[
     \Norm{f}_1:=\Norm{f}_{\Ll^1(\R,H)}:=\int_\R\Norm{f(t)} dt.
\]
\end{proposition}

\begin{proof}
Fix a Cauchy sequence $f_\nu\in \Ll^1(\R,H)$.
Motivated by the real valued case, see e.g.
Rudin~\cite[Ch.\,3]{Rudin:1987a}
or~\cite[Thm.\,4.9]{Salamon:2016a},
pick a subsequence, still denoted by $f_\nu$, such that
each difference $\Norm{f_{\nu}-f_{\nu+1}}_1$ is less than $2^{-\nu}$.
That is
\begin{equation}\label{eq:nkjnkj}
     \Norm{f_{\nu}-f_{\nu+1}}_{\Ll^1(\R,H)}
     =\int_\R \underbrace{\Norm{f_{\nu}(t)-f_{\nu+1}(t)}}_{=: g_\nu(t)} dt
     \le 2^{-\nu}, \quad \nu \in\N.
\end{equation}

%\vspace{.1cm}
\noindent
\textbf{Claim 1.} 
The infinite sum of the $g_\nu$'s is finite outside a null set.
More precisely, there is a %$\Aa-\Bb$ measurable 
function $g:\R\to [0,\infty)$ and a Lebesgue null set $N\subset \R$
such that
\[
     g=\sum_{\nu=1}^\infty g_\nu,\quad\text{on $\R\setminus N$}.
\]

\begin{proof}[Proof of Claim 1]
Setting $G_n:=\sum_{\nu=1}^n g_\nu$
we obtain a pointwise monotone sequence
$G_n\le G_{n+1}$ since the $g_\nu$ are non-negative.
Define $G:\R\to[0,\infty]$ by
\[
     G(t):=\sum_{\nu=1}^\infty g_\nu(t),\quad\text{for $t\in\R$}.
\]
The Lebesgue monotone convergence theorem,
see e.g.~\cite[Thm.\,1.37]{Salamon:2016a}, asserts that the function
$G$ is measurable and provides the first step in
\[
     \int_\R G(t)\, dt=\lim_{n\to\infty}\int_\R G_n(t)\, dt
     =\lim_{n\to\infty} \int_\R \sum_{\nu=1}^n g_\nu(t)\, dt
     =\sum_{\nu=1}^\infty \underbrace{\int_\R g_\nu(t)\, dt}
     _{\le 2^{-\nu}}
    \le 1.
\]
The inequality uses~(\ref{eq:nkjnkj}).
By finiteness of the integral $G$ can only take on an infinite value
on a set $N$ of measure zero. Consequently the function defined by
\[
     g(t):=
     \begin{cases} 
       G(t)&\text{, $t\in\R\setminus N$,}\\
       0)&\text{, $t\in N$,}
     \end{cases}
\]
is pointwise finite. It is also measurable.
\end{proof}

%\vspace{.1cm}
\noindent 
\textbf{Claim 2.} 
Outside the null set $N$ from Claim~1, that is for $t\in\R\setminus
N$, the sequence $f_\nu(t)$ is Cauchy in $H$.

\begin{proof}[Proof of Claim 2]
Given $t\in\R\setminus N$, pick $\eps>0$. Because
$G(t)=\sum_{\nu=1}^\infty g_\nu(t)<\infty$ is finite, there is an index
$\nu_0=\nu_0(\eps)$ such that $\sum_{\nu=\nu_0}^\infty g_\nu(t)<\eps$.
Pick $\nu_2\ge\nu_1\ge\nu_0$. We estimate
\begin{equation*}
\begin{split}
     \Norm{\left(f_{\nu_2}-f_{\nu_1}\right) (t)}
   &=\Norm{\sum_{\nu=\nu_1}^{\nu_2-1}\left(f_{\nu+1}-f_{\nu}\right)(t)}\\
   &\le\sum_{\nu=\nu_1}^{\nu_2-1}
     \underbrace{\Norm{\left(f_{\nu+1}-f_{\nu}\right)(t)}}_{g_\nu(t)}\\
   &\le\sum_{\nu=\nu_0}^\infty g_\nu(t)\\
   &<\eps.
\end{split}
\end{equation*}
This proves Claim 2.
\end{proof}

Because $H$ is complete it follows from Claim 2
that for all $t\in\R\setminus N$ the limit
$\lim_{\nu\to\infty} f_\nu(t) \in H$ exists.
We obtain a function $f:\R\to H$ by defining
\[
     f(t):=
     \begin{cases} 
       \lim_{\nu\to\infty} f_\nu(t)&\text{, $t\in\R\setminus N$,}\\
       0&\text{, $t\in N$.}
     \end{cases}
\]
By Theorem~\ref{thm:Pettis} of Pettis measurability of $f:\R\to H$
is equivalent to the following claim.

\vspace{.2cm}
\noindent
\textbf{Claim 3.} (Measurability of $f:\R\to H$)
The function $\varphi_x:=\inner{f}{x} :(\R,\Aa)\to(\R,\Bb)$
is measurable $\forall x\in H$.

\begin{proof}[Proof of Claim 3]
Define the function $\varphi_{x,\nu}:\R\to\R$ by
$t\mapsto\inner{x}{f_\nu(t)}$.
Note that for every $t\in\R\setminus N$
one has $\lim_{\nu\to\infty}\varphi_{x,\nu}(t)=\varphi_{x}(t)$,
because $f_\nu(t)\to f(t)$ by definition of $f$.
Therefore $\varphi_x$ is up to the set $N$ of measure zero
the pointwise limit of a sequence of measurable functions
and hence itself measurable.
\end{proof}

\vspace{.2cm}
\noindent
\textbf{Claim 4.} (Convergence)
$\lim_{\nu\to\infty}\norm{f-f_\nu}_{\Ll^1(\R,H)}=0$.

\begin{proof}[Proof of Claim 4]
Given $\eps>0$, choose $\nu$ such that
$1/2^{\nu-1}<\eps$. 
Using Fatou's Lemma to obtain the first inequality
we estimate
\begin{equation*}
\begin{split}
     \Norm{f-f_\nu}_{\Ll^1(\R,H)}
   &=\int_\R \Norm{f_\nu(t)-f(t)} dt\\
   &=\int_\R \liminf_{k\to\infty} \Norm{f_\nu(t)-f_k(t)} dt\\
   &\le\liminf_{k\to\infty} \int_\R
     \underbrace{\Norm{f_\nu(t)-f_k(t)}}_{\le \sum_{j=\nu}^{k-1} \Norm{f_j(t)-f_{j+1}(t)}} dt\\
%\end{split}
%\end{equation*}
%\begin{equation*}
%\begin{split}
   &\le\liminf_{k\to\infty} \sum_{j=\nu}^{k-1} 
     \underbrace{
                         \int_\R\overbrace{\Norm{f_j(t)-f_{j+1}(t)}}^{=g_j(t)} dt
                         }_{\le 1/2^j}
     \\
   &\le\sum_{j=\nu}^\infty\frac{1}{2^j}
     =\frac{1}{2^{\nu-1}}\\
   &<\eps.
\end{split}
\end{equation*}
This proves Claim~4.
\end{proof}

By Claim~4 the limit $f$ is in $\Ll^1(\R,H)$
and $f_\nu\to f$ in $\Ll^1(\R,H)$.
This proves Proposition~\ref{prop:Ll(R,H)-complete}.
\end{proof}

On $\Ll^1(\R,H)$ consider the equivalence relation
$f\sim g$ if the two maps are equal outside a set of measure zero.
On the quotient space 
\[
     L^1(\R,H):=\Ll^1(\R,H)/\sim
\]
the semi-norm $\norm{\cdot}_1$ is a norm. Hence $L^1(\R,H)$
is a Banach space by Proposition~\ref{prop:Ll(R,H)-complete}.
By abuse of notation we denote the elements of $L^1(\R,H)$ still~by~$f$.

%%%%%%% Subsection: Lagrangian boundary conditions %%%%%%%
%%%%%%%%%%%%%%%%%%%%%%%%%%%%%%%%%%%%%%%%%%%%
\subsection*{The Banach spaces \boldmath$L^p(\R,H)$}

Similarly for $p\in(1,\infty)$ one calls a Hilbert space valued
function $f:\R\to H$ \textbf{\boldmath$p$-integrable}
if it satisfies~(i) in Definition~\ref{def:integrable}
and~(ii) is replaced by finiteness of the $p$-semi-norm
\[
     \Norm{f}_p:=\Norm{f}_{\Ll^p(\R,H)}
    :=\left(\int_\R\Norm{f(t)}^p dt\right)^{\frac{1}{p}}.
\]
Let $\Ll^p(\R,H)$ be the set of all $p$-integrable
functions $f:\R\to H$.
As in Propositions~\ref{prop:vector-space}
and~\ref{prop:Ll(R,H)-complete} one shows that
$\Ll^p(\R,H)$ is a vector space which is complete with respect
to the $p$-semi-norm.
On the quotient space 
\[
     L^p(\R,H):=\Ll^p(\R,H)/\sim
\]
the semi-norm $\norm{\cdot}_p$ is a norm. Hence $L^p(\R,H)$
is a Banach space. Again we denote the elements $[f]$ of $L^p(\R,H)$
still by $f$.

%%%%%%% Subsection: Lagrangian boundary conditions %%%%%%%
%%%%%%%%%%%%%%%%%%%%%%%%%%%%%%%%%%%%%%%%%%%%
\subsection*{The Sobolev space \boldmath$W^{1,p}(\R,H)$}

Fix $p\in[1,\infty)$ and let $W^{1,p}(\R,H)$ be the vector
space of all $f\in L^p(\R,H)$ for which there exists
an element $v\in L^p(\R,H)$ such that
\[
     \int_\R\INNER{f(t)}{\dot \varphi(t)} dt
     =-\int_\R\INNER{v(t)}{\varphi(t)} dt
\]
for every $\varphi\in C^\infty_{\rm c}(\R,H)$.
If such a map $v$ exists, then it is unique and called
the \textbf{weak derivative} of $f$.
We denote $v$ by the symbol $\dot f$ or $f^\prime$.
The vector space $W^{1,p}(\R,H)$ is endowed with the norm
$\norm{f}_{1,p}:=\norm{f}_p+\norm{\dot f}_p$.

\begin{proposition}\label{prop:Lp(R,H)-complete}
The space $W^{1,p}(\R,H)$ is a Banach space.
\end{proposition}

\begin{proof}
Let $f_\nu$ by a Cauchy sequence in $W^{1,p}(\R,H)$.
Hence $f_\nu$ forms a Cauchy sequence in $L^p(\R,H)$
as well as the weak derivatives $\dot f_\nu$.
By completeness of $L^p(\R,H)$
there are elements $f,v\in L^p(\R,H)$ such that
\[
     f_\nu\stackrel{L^p}{\longrightarrow} f,\qquad
     \dot f_\nu\stackrel{L^p}{\longrightarrow} v.
\]
In view of Lemma~\ref{le:bhjbjh}
we compute
\begin{equation*}
\begin{split}
     \int_\R\INNER{f(t)}{\dot\varphi(t)} dt
   &=\lim_{\nu\to\infty}\int_\R\INNER{f_\nu(t)}{\dot\varphi(t)} dt\\
   &=-\lim_{\nu\to\infty}\int_\R\INNER{\dot f_\nu(t)}{\varphi(t)} dt\\
   &=-\int_\R\INNER{v(t)}{\varphi(t)} dt.
\end{split}
\end{equation*}
This shows that $v$ is the weak derivative of $f$.
Hence $f\in W^{1,p}(\R,H)$ and $f_\nu\to f$ in $W^{1,p}(\R,H)$.
This shows completeness and proves
Proposition~\ref{prop:Lp(R,H)-complete}.
\end{proof}

\begin{lemma}\label{le:bhjbjh}
Let $p\in[1,\infty)$.
Let $f_\nu\in L^p(\R,H)$ be a sequence
that converges to an element $f\in L^p(\R,H)$ and $\varphi\in
C^\infty_{\rm c}(\R,H)$ is of compact support. Then
\[
     \int_\R\INNER{f}{\varphi} dt
     =\lim_{\nu\to\infty}\int_\R\INNER{f_\nu}{\varphi} dt.
\]
\end{lemma}

\begin{proof}
The support of $\varphi$ is contained in $[-T,T$ for $T>0$
sufficiently large. Moreover, there is a constant $c>0$ such that
$\norm{\varphi(t)}\le c$ for every $t\in\R$. Let $q$ be such that
$1/p+1/q=1$. We estimate
\begin{equation*}
\begin{split}
     \Abs{\int_\R\INNER{f(t)-f_\nu(t)}{\varphi} dt}
   &=\Abs{\int_{-T}^T\INNER{f(t)-f_\nu(t)}{\varphi(t)} dt}\\
   &\le\int_{-T}^T\Norm{f(t)-f_\nu(t)}\cdot\Norm{\varphi(t)} dt\\
   &\le c\int_{-T}^T 1\cdot\Norm{f(t)-f_\nu(t)} dt\\
   &\le c (2T)^{1/q}
     \left(\int_{-T}^T \Norm{f(t)-f_\nu(t)}^p dt\right)^{1/p}\\
   &\le c (2T)^{1/q}\Norm{f-f_\nu}_{L^p(\R,H)}.
\end{split}
\end{equation*}
The first inequality uses the Cauchy-Schwarz inequality
in the Hilbert space $H$.
The third inequality uses H\"older for real valued functions.
\end{proof}

%%%%%%% Subsection: Lagrangian boundary conditions %%%%%%%
%%%%%%%%%%%%%%%%%%%%%%%%%%%%%%%%%%%%%%%%%%%%
\subsection*{The Sobolev spaces \boldmath$W^{k,p}(\R,H)$}

Recursively, for $k\in\N$,
we define $W^{k+1,p}(\R,H)$
to be the space of all functions $f\in W^{1,p}(\R,H)$
whose weak derivative $\dot f$ lies in $W^{k,p}(\R,H)$.
The vector space $W^{k+1,p}(\R,H)$ is endowed with the norm
$\norm{f}_{k+1,p}:=\norm{f}_p+\norm{\dot f}_{k,p}$.
Using the argument in the proof of
Proposition~\ref{prop:Lp(R,H)-complete}
inductively we obtain that $W^{k+1,p}(\R,H)$ is a Banach space.

\begin{proposition}\label{prop:Wkp(R,H)-complete}
The space $W^{k,p}(\R,H)$ is a Banach space
whenever $k\in\N_0$ and $p\in[1,\infty)$.
\end{proposition}

%%%%%%%%%%%%%%%%%%%%%%%%%%%%%%%%%%%
%%%%%%%%%% BACKMATTER %%%%%%%%%%%%%%%
%%%%%%%%%%%%%%%%%%%%%%%%%%%%%%%%%%%
%\backmatter
%\include{glossary}
%\include{solutions}
%\printindex
%
%\end{document}

%\backmatter %commands for bibliography, index

%%%%%%%%%%%%%%%%%%%%%%%%%
%%%%%%%%% REFERENCES %%%%%%
%%%%%%%%%%%%%%%%%%%%%%%%
%\renewcommand{\bibname}{References}
%\bibliographysty le{plain}
         %   erzeugt:     [1] Joa Weber
%\bibliographystyle{abbrv}
         %  erzeugt:      [1] J. Weber and 
\bibliographystyle{alpha}
         %  article:    [Web05]  J. Weber
         %  book:      [Web05]  Joa Weber
         % more authors: [HZ87]
%%%%%%%%%%%%%%%%%%%%%%%%%
%% include Bibliography in TOC %%
% en.wikibooks.org/wiki/LaTeX/Bibliography_Management#Using_tocbibind
%%%%%%%%%%%%%%%%%%%%%%%%%
% Using hyperref, one should say:
%\cleardoublepage
%\phantomsection
\addcontentsline{toc}{section}{References}
\bibliography{$HOME/Dropbox/0-Libraries+app-data/Bibdesk-BibFiles/library_math}{}

\begin{thebibliography}{FFGW16}

\bibitem[AF03]{Adams:2003a}
Robert~A. Adams and John J.~F. Fournier.
\newblock {\em Sobolev spaces}, volume 140 of {\em Pure and Applied Mathematics
  (Amsterdam)}.
\newblock Elsevier/Academic Press, Amsterdam, second edition, 2003.

\bibitem[AF13]{albers:2013a}
Peter Albers and Urs Frauenfelder.
\newblock Exponential decay for sc-gradient flow lines.
\newblock {\em J. Fixed Point Theory Appl.}, 13(2):571--586, 2013.

\bibitem[AFS18a]{Albers:2018c}
P.~{Albers}, U.~{Frauenfelder}, and F.~{Schlenk}.
\newblock {A compactness result for non-local unregularized gradient flow
  lines}.
\newblock {\em ArXiv e-prints}, February 2018.

\bibitem[AFS18b]{Albers:2018b}
P.~{Albers}, U.~{Frauenfelder}, and F.~{Schlenk}.
\newblock {What might a Hamiltonian delay equation be?}
\newblock {\em ArXiv e-prints}, February 2018.

\bibitem[AW13]{Albers:2013b}
Peter Albers and Kris Wysocki.
\newblock {M-polyfolds in Morse theory}.
\newblock Unpublished manuscript, March 2013.

\bibitem[Coh93]{Cohn:1993a}
Donald~L. Cohn.
\newblock {\em Measure theory}.
\newblock Birkh\"auser Boston, Inc., Boston, MA, 1993.
\newblock Reprint of the 1980 original.

\bibitem[DT98]{donaldson:1998a}
S.~K. Donaldson and R.~P. Thomas.
\newblock Gauge theory in higher dimensions.
\newblock In {\em The geometric universe ({O}xford, 1996)}, pages 31--47.
  Oxford Univ. Press, Oxford, 1998.

\bibitem[Eva98]{evans:1998a}
Lawrence~C. Evans.
\newblock {\em Partial differential equations}, volume~19 of {\em Graduate
  Studies in Mathematics}.
\newblock American Mathematical Society, Providence, RI, 1998.

\bibitem[FFGW16]{Fabert:2016a}
Oliver Fabert, Joel~W. Fish, Roman Golovko, and Katrin Wehrheim.
\newblock Polyfolds: a first and second look.
\newblock {\em EMS Surv. Math. Sci.}, 3(2):131--208, 2016.

\bibitem[Flo88a]{floer:1988a}
Andreas Floer.
\newblock Morse theory for {L}agrangian intersections.
\newblock {\em J. Differential Geom.}, 28(3):513--547, 1988.

\bibitem[Flo88b]{floer:1988c}
Andreas Floer.
\newblock The unregularized gradient flow of the symplectic action.
\newblock {\em Comm. Pure Appl. Math.}, 41(6):775--813, 1988.

\bibitem[Flo89]{floer:1989a}
Andreas Floer.
\newblock Symplectic fixed points and holomorphic spheres.
\newblock {\em Comm. Math. Phys.}, 120(4):575--611, 1989.

\bibitem[GH12]{Ginzburg:2012a}
Viktor~L. Ginzburg and Doris Hein.
\newblock Hyperk\"ahler {A}rnold conjecture and its generalizations.
\newblock {\em Internat. J. Math.}, 23(8):1250077, 15, 2012.

\bibitem[HNS09]{hohloch:2009a}
Sonja Hohloch, Gregor Noetzel, and Dietmar~A. Salamon.
\newblock Hypercontact structures and {F}loer homology.
\newblock {\em Geom. Topol.}, 13(5):2543--2617, 2009.

\bibitem[Hof17]{Hofer:2017c}
Helmut H.~W. Hofer.
\newblock Polyfolds and {F}redholm theory.
\newblock In {\em Lectures on geometry}, Clay Lect. Notes, pages 87--158.
  Oxford Univ. Press, Oxford, 2017.

\bibitem[HWZ07]{Hofer:2007a}
H.~Hofer, K.~Wysocki, and E.~Zehnder.
\newblock A general {F}redholm theory. {I}. {A} splicing-based differential
  geometry.
\newblock {\em J. Eur. Math. Soc. (JEMS)}, 9(4):841--876, 2007.

\bibitem[HWZ17]{Hofer:2017a}
H.~{Hofer}, K.~{Wysocki}, and E.~{Zehnder}.
\newblock {Polyfold and Fredholm Theory}. 714 pages, %.
\newblock {\em ArXiv e-prints}, July 2017.

\bibitem[{Kan}11]{Kang:2011a}
J.~{Kang}.
\newblock {Local invariant for scale structures on mapping spaces}.
\newblock {\em ArXiv e-prints}, August 2011.

\bibitem[{Kre}15]{Kreuter:2015a}
Marcel {Kreuter}.
\newblock {Sobolev Spaces of Vector-Valued Functions}.
\newblock Master's thesis, Universit{\"a}t Ulm, April 2015.

\bibitem[Pet38]{Pettis:1938a}
B.~J. Pettis.
\newblock On integration in vector spaces.
\newblock {\em Trans. Amer. Math. Soc.}, 44(2):277--304, 1938.

\bibitem[RS95]{robbin:1995a}
Joel {Robbin} and Dietmar {Salamon}.
\newblock {The spectral flow and the Maslov index}.
\newblock {\em {Bull. Lond. Math. Soc.}}, 27(1):1--33, 1995.

\bibitem[Rud87]{Rudin:1987a}
Walter Rudin.
\newblock {\em Real and complex analysis}.
\newblock McGraw-Hill Book Co., New York, third edition, 1987.

\bibitem[Sal16]{Salamon:2016a}
Dietmar~A. Salamon.
\newblock {\em Measure and integration}.
\newblock EMS Textbooks in Mathematics. European Mathematical Society (EMS),
  Z\"urich, 2016.

\bibitem[{Sim}14]{Simcevic:2014a}
Tatjana {Sim\v{c}evi\'{c}}.
\newblock {\em {A Hardy Space Approach to Lagrangian Floer Gluing}}.
\newblock PhD thesis, ETH Z\"urich, October 2014.

\bibitem[SW06]{salamon:2006a}
Dietmar Salamon and Joa Weber.
\newblock Floer homology and the heat flow.
\newblock {\em Geom. Funct. Anal.}, 16(5):1050--1138, 2006.

\bibitem[Tri78]{Triebel:1978a}
H.~Triebel.
\newblock {\em Interpolation theory, function spaces, differential operators}.
\newblock VEB Deutscher Verlag der Wissenschaften, Berlin, 1978.

\bibitem[Vit98]{Viterbo:1998a}
Claude Viterbo.
\newblock {Functors and computations in Floer homology with applications, II.}
\newblock {Preprint Universit{\'e} Paris-Sud no. 98-15}, 1998.

\bibitem[Web13a]{weber:2013b}
Joa Weber.
\newblock {M}orse homology for the heat flow.
\newblock {\em Math. Z.}, 275(1-2):1--54, 2013.

\bibitem[Web13b]{weber:2013a}
Joa Weber.
\newblock {M}orse homology for the heat flow -- {L}inear theory.
\newblock {\em Math. Nachr.}, 286(1):88--104, 2013.

\bibitem[{Weh}12]{Wehrheim:2012b}
K.~{Wehrheim}.
\newblock {Fredholm notions in scale calculus and Hamiltonian Floer theory}.
\newblock {\em ArXiv e-prints}, September 2012.

\bibitem[Yos95]{Yosida:1995a}
K\=osaku Yosida.
\newblock {\em Functional analysis}.
\newblock Classics in Mathematics. Springer-Verlag, Berlin, 1995.
\newblock Reprint of the sixth (1980) edition.

\end{thebibliography}
%$
%%%%%%%%%%%%%%%%%%%%%%%%%
%%%%%%%%% standard %%%%%%%%%
%%%%%%%%%%%%%%%%%%%%%%%%%
%\begin{thebibliography}{00000}
%\small
%\end{thebibliography}

%%%%%%%%%%%%%%%%%%%%%%%%%%%%%%%%%%%%
%%%%%%%%%%%%% GLOSSARY %%%%%%%%%%%%%%%
%%%%%%%%%%%%%%%%%%%%%%%%%%%%%%%%%%%%
% Using hyperref, one should say:
%\cleardoublepage
%\phantomsection
%\printnomenclature
%
%This is $F$\label{nomen:F} 
%\nomenclature[EF]{$F$}{Objective function}{}{\pageref{nomen:F}}
%\clearpage

%%%%%%%%%%%%%%%%%%%%%%%%%
%%%%%%%%% INDEX %%%%%%%%%%
%%%%%%%%%%%%%%%%%%%%%%%%%
% Using hyperref, one should say:
%\cleardoublepage
%\phantomsection
%\addcontentsline{toc}{chapter}{Index}
%\printindex

\end{document}